\documentclass{article}

\usepackage{microtype}
\usepackage{graphicx}
\usepackage{subfigure}
\usepackage{booktabs} 

\usepackage{graphicx,verbatim,lineno,titletoc}
\usepackage{amssymb,mathrsfs}
\usepackage{amsmath}
\usepackage{adjustbox}
\usepackage{calc}
\usepackage{ytableau}
\usepackage{array}
\usepackage{tikz}
\usetikzlibrary{shapes.multipart,patterns,arrows}
\usepackage{appendix}

\usepackage[all]{xy}
\usepackage{enumerate}
\usepackage[english]{babel}
\usepackage{multirow}
\usepackage{float}
\usepackage{enumitem}
\usepackage{accents}
\usepackage{mathtools}
\mathtoolsset{showonlyrefs=true}
\usepackage{wrapfig}
\usepackage{pifont}

\usepackage{hyperref}

\usepackage[accepted]{icml2024}

\usepackage{amsmath}
\usepackage{amssymb}
\usepackage{mathtools}
\usepackage{amsthm}

\theoremstyle{plain}
\newtheorem{theorem}{Theorem}[section]

\newtheorem{lemma}[theorem]{Lemma}
\newtheorem{corollary}[theorem]{Corollary}
\theoremstyle{definition}
\newtheorem{definition}[theorem]{Definition}

\theoremstyle{remark}

\theoremstyle{prop}
\newtheorem{prop}[theorem]{Proposition}
\newtheorem{example}[theorem]{Example}

\newtheorem{innercustomassumption}{}
\newenvironment{customassumption}[1]
{\renewcommand\theinnercustomassumption{#1}\innercustomassumption}
{\endinnercustomassumption}
\usepackage[textsize=tiny]{todonotes}

\def\R{\mathbb{R}}

\def\eps{\varepsilon}

\def\X{\mathbf{X}}

\def\param{\boldsymbol{\theta}}
\def\Param{\boldsymbol{\Theta}}

\definecolor{hancolor}{rgb}{0.1, 0.0, 0.9}
\definecolor{YLcolor}{rgb}{0.8, 0.1, 0.1}

\DeclareMathOperator*{\argmin}{arg\,min}

\DeclareMathOperator*{\argmax}{arg\,max}

\DeclareMathOperator{\rtr}{\textup{Rtr}}
\DeclareMathOperator{\rinj}{r_{\textup{inj}}}

\DeclareMathOperator{\rrtr}{r_{\textup{Rtr}}}

\def\M{\mathcal{M}}
\newcommand{\Exp}{\operatorname{Exp}}
\newcommand{\Proj}{\operatorname{Proj}}
\newcommand{\grad}{\operatorname{grad}}

\icmltitlerunning{Convergence and Complexity Guarantee for  Inexact First-order Riemannian Optimization Algorithms}

\begin{document}

\twocolumn[
\icmltitle{Convergence and Complexity Guarantee for\\  Inexact First-order Riemannian Optimization Algorithms}



\icmlsetsymbol{equal}{*}

\begin{icmlauthorlist}
\icmlauthor{Yuchen Li}{wisc}
\icmlauthor{Laura Balzano}{umich}
\icmlauthor{Deanna Needell}{ucla}
\icmlauthor{Hanbaek Lyu}{wisc}
\end{icmlauthorlist}

\icmlaffiliation{wisc}{Department of Mathematics, University of Wisconsin-Madison, WI 53706, USA}
\icmlaffiliation{umich}{Department of Electrical Engineering and Computer Science, University of Michigan, Ann Arbor, MI 48109, USA}
\icmlaffiliation{ucla}{Department of Mathematics, University of California, 
		Los Angeles, CA 90025, USA}

\icmlcorrespondingauthor{Hanbaek Lyu}{hlyu@math.wisc.edu}

\icmlkeywords{Machine Learning, ICML}

\vskip 0.3in
]



\printAffiliationsAndNotice{}  

\begin{abstract}

We  analyze inexact Riemannian gradient descent (RGD) where Riemannian gradients and retractions are inexactly (and cheaply) computed. Our focus is on understanding when inexact RGD converges and what is the complexity in the general nonconvex and constrained setting. We answer these questions in a general framework of tangential Block Majorization-Minimization (tBMM). We establish that tBMM converges to an $\epsilon$-stationary point within $O(\epsilon^{-2})$ iterations. Under a mild assumption, the results still hold when the subproblem is solved inexactly in each iteration provided the total optimality gap is bounded. Our general analysis applies to a wide range of classical algorithms with Riemannian constraints including inexact RGD and proximal gradient method on Stiefel manifolds. We numerically validate that tBMM shows improved performance over existing methods when applied to various problems, including nonnegative tensor decomposition with Riemannian constraints, regularized nonnegative matrix factorization, and low-rank matrix recovery problems.
\end{abstract}

\section{Introduction}
\label{sec:intro}

A typical formulation of constrained Riemannian optimization takes the following form:
	\begin{align}\label{eq:def_CROPT}
		\min_{\param\in \Param \subseteq \mathcal{M}}  \left(f(\param):=\varphi(\param) + \psi(\param)\right),
	\end{align}
	where $\mathcal{M}$ is a smooth Riemannian manifold embedded in a Euclidean space, $\Param$ is a closed subset of $\mathcal{M}$,  and $f:\mathcal{M}\rightarrow \R$ is an objective function consisting of a smooth part $\varphi$ and convex (and possibly nonsmooth) part $\psi$. Riemannian optimization problems of the form \eqref{eq:def_CROPT} have a wide array of applications ranging from the computation of linear algebraic quantities and factorizations and problems with nonlinear differentiable constraints to the analysis of shape space and automated learning \cite{ring2012optimization,jaquier2020bayesian}. These applications arise either because of implicit constraints on the problem to be optimized or because the domain is naturally defined as a manifold.

\textbf{Motivation: Inexact RGD.} In many Riemannian optimization problem instances, the dimension of the parameter space is much less than that of the ambient dimension. Such `latent' low-dimensional structure in the parameter space can be utilized by using the  Riemannian gradient of the objective that lives in the (low-dimensional) tangent space, instead of the full gradient in the (high-dimensional) ambient space. Thus many Riemannian optimization methods take the following form \cite{baker2008implicit,yang2007globally, boumal2019global,edelman1998geometry}: Iteratively, 
\vspace{-0.2cm}
\begin{description}[itemsep=0.1cm]
    \item[(i)] Compute descent direction in the tangent space;
    \item[(ii)] Take a step in that direction along a geodesic. 
\end{description}
\vspace{-0.2cm}
However, step \textbf{(ii)} is often challenging in practice so other approaches alleviate this burden by utilizing approximations or imposing additional assumptions \cite{absil2008optimization,boumal2023introduction}. A popular way to implement a similar idea with less computational burden for computing the geodesic is to make the parameter update first in the tangent spaces and then map the resulting point back onto the manifold using a retraction (which is like a projection from the tangent space onto the manifold). 

Perhaps the simplest and most widely used Riemannian optimization algorithms for smooth objectives $f$ of the above form is Riemannian gradient descent (RGD): 
\begin{align}\label{eq:RGD}
\hspace{-1cm}\textbf{(RGD)}\qquad 
\begin{cases}
    V_{n} &\leftarrow -\grad f(\param_{n-1}) \\
    \param_{n} &\leftarrow  \rtr_{\param_{n-1}}\left( \alpha_{n} V_{n} \right). 
\end{cases}
\end{align}
Here $\grad f(\param_{n-1})$ denotes the Riemannian gradient of $f$ at $\param_{n-1}$, $\rtr_{\param_{n-1}}$ denotes the retraction map from the tangent space $T_{\param_{n-1}}\mathcal{M}$ at base point $\param_{n-1}$ onto the manifold $\mathcal{M}$ (see Appendix \ref{sec:preliminaries_Riemannian}), and $\alpha_{n}>0$ is a step size. In the literature, one typically assumes that computing Riemannian gradients and retractions are computationally feasible. However, there are several problem instances where either computing exact Riemannian gradient or the retraction is difficult \cite{wiersema2023optimizing,ablin2022fast,wang2021no}. For such situations, it is reasonable to consider the following `inexact version' of RGD:
\begin{align}\label{eq:RGD_inexact}
\hspace{-1cm}\textbf{(Inexact RGD)}\qquad 
\begin{cases}
    \hat{V}_{n} &= -\widehat{\grad} f(\param_{n-1}) \\
    \param_{n} &=  \widehat{\rtr}_{\param_{n-1}}\left( \alpha_{n} \hat{V}_{n} \right),
\end{cases}
\end{align}
where $\widehat{\grad}$ and $\widehat{\rtr}$ are computationally feasible inexact Riemannian gradient and retraction operators, respectively. 
Our main question is the following. \textit{ When does this inexact RGD converge? What can we say about its complexity?} We answer these questions in a general framework of Riemannian tangential Block Majorization-Minimization (tBMM). Here we state a corollary of our general result for the context of inexact RGD (see proof in Appendix \ref{sec:proof_intro}). 

\begin{corollary}[Inexact RGD]\label{cor:inexact_RGD}
Let $(\param_n)_n$ be the iterations generated by the inexact RGD \eqref{eq:RGD_inexact} for solving \eqref{eq:def_CROPT} with $\psi=0$. Assume the manifold is complete 
and the objective function is uniformly lower bounded with compact sub-level sets. Then each limit point of $(\param_n)_n$ is a stationary point and an $\eps$-stationary point is obtained within $O(\eps^{-2})$ iterations if the following holds: 
\vspace{-0.3cm}
\begin{description}[itemsep=-0.1cm]
    \item[(i)] $\nabla \varphi$ is $L$-Lipschitz continuous for some $L>0$. 
    \item[(ii)] For each $n$, choose $V_n\in T_{\param_{n-1}}\mathcal{M}$ such that 
\begin{align}
\rtr_{\param_{n-1}}(V_n) = \widehat{\rtr}_{\param_{n-1}}\left( \hat{V}_{n} \right).
\end{align}
Denote $g_n(\eta):= \langle \grad f(\param_{n-1}),\eta\rangle + \frac{L}{2}\|\eta\|^2$ and define the optimality gap as 
    \begin{align}
        \Delta_{n}:= g_n(\alpha_{n} V_n)-g_n (-\frac{1}{2L}\grad f(\param_{n-1})).
    \end{align}
    Then $\sum_{n=0}^{\infty}\Delta_n <\infty$.
\end{description}
\end{corollary}

\textbf{General framework through tangential MM.} 
We establish Corollary \ref{cor:inexact_RGD} for the broader class of first-order Riemannian optimization algorithms called \textit{tangential Majorization-Minimization} (tMM). We first recall that the classical Majorization-Minimization algorithm generalizes gradient descent in Euclidean space \cite{mairal2013optimization}:

\vspace{-0.6cm}
{\small
\begin{align}
	\param_{n} - \param_{n-1} \leftarrow 
		&=  - \frac{1}{\lambda} \nabla f(\param_{n-1}) \label{eq:BProxLinear} \\
        &\hspace{-1.8cm}=\argmin_{\eta\in \R^{p}} \left[ g_{n}(\eta):= f(\param_{n-1} ) + \langle \nabla f(\param_{n-1}) ,\, \eta  \rangle + \frac{\lambda}{2} \lVert \eta \rVert^{2} \right]. \nonumber 
\end{align}
}

\noindent Assuming $\lambda\ge L$, the Lipschitz parameter for $\nabla f$,  the quadratic function $g_{n}$ above satisfies (majorization) $g_{n}(\eta)\ge f(\param_{n}+\eta)$ and (tightness) $\nabla g_{n}(0)=\nabla f(\param_{n-1})$ and is called the prox-linear surrogate of $f$ at $\param_{n-1}$. Thus, we majorize objective $f$ near $\param_{n-1}$ by the prox-linear surrogate $g_{n}$ and minimize it to find a descent direction.  

In the Riemannian setting, RGD can be thought of as a Riemannian version of MM. Notice that the prox-linear surrogate $g_{n}$ above is in fact defined on the tangent space of $\R^{p}$ at $\param_{n-1}$, and what it majorizes is not the original objective $f$, but the `pull-back' objective $f(\param_{n-1}+\eta)$ defined on the tangent space. Applying this observation to RGD, we seek to majorize the pull-back objective 
\begin{align}\label{def:pull_back_objective}
\hat{\varphi}_{n} := \varphi_{n}\circ \rtr_{\theta_{n-1}}: T_{\param_{n-1}}\mathcal{M} \rightarrow \R
\end{align}
obtained by precomposing the objective function $\varphi:\mathcal{M}\rightarrow \R$ with the retraction at $\param_{n-1}$. In this way, the Riemannian objective is now lifted to the Euclidean objective $\hat{\varphi}_{n}$ on the tangent space. Here we apply the usual MM strategy on the tangent space to find a descent direction, take a step on the tangent space, and retract back onto the manifold. We call this procedure `tangential MM', which is concisely stated below:
\begin{align}
			\hat{g}_{n} &\leftarrow \left[ \begin{matrix}\textup{Majorizing surrogate of $\hat{\varphi}_{n}$ s.t. $\hat{g}_{n}(\mathbf{0})=\hat{\varphi}(\mathbf{0})$}\end{matrix} \right]\vspace{0.3cm}  \nonumber \\
            V_{n} &\leftarrow \argmin_{\eta\in T_{\param_{n-1}}\mathcal{M}} \hat{g}_{n}(\eta) \label{eq:tMM} \\  
            \param_{n} &\leftarrow \rtr_{\param_{n-1}}\left(  \alpha_{n}V_{n}  \right). \nonumber
	\end{align}
To view the inexact RGD for smooth objectives \eqref{eq:RGD_inexact} as a special case of tMM \eqref{eq:tMM}, let 
\begin{align}\label{eq:intro_g_n}
    \hat{g}_n(\eta):= \varphi(\param_{n-1})+\langle \grad \varphi(\param_{n-1}),\eta\rangle + \frac{\lambda}{2}\|\eta\|^2
\end{align}
 where $\lambda\ge L$ and $L>0$ is the Lipschitz continuity parameter of $\nabla \varphi$. Then one can verify the update of tMM using \eqref{eq:intro_g_n} is the same as \eqref{eq:RGD_inexact} (see Section \ref{sec:brpl} for details).
 

In this work, we consider the more general setting when the manifold $\mathcal{M}$ is a product manifold given by $\mathcal{M} = \M^{(1)}\times\cdots\times\M^{(m)}$. Then problem \eqref{eq:def_CROPT} becomes a multi-block Riemannian constrained problem as follows:
\begin{align}\label{eq:def_CROPT_block}
		\min_{ \substack{\param=[\theta^{(1)},\dots,\theta^{(m)}] \\ \theta^{(i)} \in \Theta^{(i)} \subseteq \mathcal{M}^{(i)} \,\, \textup{for $i=1,\dots,m$}}  } \hspace{-0.7cm}\left(f(\param)= \varphi(\param) + \psi(\param) \right).
	\end{align} 

Given that the problem \eqref{eq:def_CROPT_block} is typically nonconvex, expecting an algorithm to converge to a globally optimal solution from an arbitrary initialization might not always be reasonable. Instead, our goal is to ensure global convergence to stationary points from any initialization. In certain problem classes, stationary points can be practically and theoretically as good as global optimizers  \cite{mairal2010online, sun2015nonconvex}. Additionally, determining the iteration complexity of such algorithms is crucial for both theoretical and practical purposes. This involves bounding the worst-case number of iterations needed to achieve an $\epsilon$-approximate stationary point (appropriately defined in Sec. \ref{sec:optimality_measures}).


In this work, we propose a block-extension of tMM in \eqref{eq:tMM} that can also handle additional (geodesically convex) constraints within each manifold $\mathcal{M}^{(i)}$ as well as nonsmooth nonconvex objectives (see Algorithm \ref{algorithm:BMM}).
 We carefully analyze tBMM in various settings and obtain first-order optimality guarantees and iteration complexity under inexact computations. From the general results, we can easily deduce Corollary \ref{cor:inexact_RGD}.
 
\subsection{Related Works}
 Under the Euclidean setting, i.e. when each $\M^{(i)}$ in \eqref{eq:def_CROPT_block} is a Euclidean space, the corresponding Euclidean block MM method has been well studied in the literature. For convex problems, the Euclidean block MM method is studied in \cite{xu2013block} with prox-linear surrogates, and in \cite{razaviyayn2013unified,hong2015unified} with general surrogates. For nonconvex problems, some variants of block MM are studied in some recent works, including BMM-DR in \cite{lyu2023block}, and BCD-PR in \cite{kwon2023complexity}. 

 Under the Riemannian setting, some recent work showed convergence and complexity of block MM methods for solving \eqref{eq:def_CROPT_block} when $\psi=0$. Namely, in \cite{gutman2023coordinate}, the authors established a sublinear convergence rate for an block-wise Riemannian gradient descent. However, the retraction considered there is restricted to the exponential map, which excludes many commonly used retractions in the literature. In \cite{peng2023block}, the authors established convergence and complexity results of general block MM on compact manifolds. In \cite{li2023convergence}, convergence and complexity results are established for the Riemannian block MM methods on general Riemannian manifolds. Moreover, extra constrained sets on the manifolds and inexact computation of subproblems are allowed in the general framework of \cite{li2023convergence}. However, all these analyses are limited to the smooth problem ($\psi=0$) and cannot be directly applied to the general problem \eqref{eq:def_CROPT_block}.

 For nonconvex nonsmooth problems, in \cite{chen2020proximal}, the authors studied a tangential type of MM method with prox-linear surrogates on Stiefel manifolds with complexity guarantees. However, the problem considered there is a single-block problem.

\subsection{Our Contributions}
In this work, we propose tBMM, which is a general framework of tangential type Riemannian block MM algorithms for solving nonconvex, nonsmooth, multi-block constrained Riemannian optimization problems \eqref{eq:def_CROPT_block}, and allowing inexact computation of subproblems. We thoroughly analyze tBMM \eqref{eq:tMM} and obtain asymptotic convergence to the set of stationary points and iteration complexity. The theoretical contributions of this work, compared to the aforementioned related work, lie especially in the following three aspects,
\begin{description}[itemsep=0.01cm]
    \item[(1)] (Iteration complexity) tBMM is applicable to nonsmooth, nonconvex, multi-block Riemannian optimization problems, and we derive the iteration complexity of $O(\eps^{-2})$ along with asymptotic convergence to the set of stationary points. See Theorem \ref{thm:RBMM_MtMt}.
    \item[(2)] (Constrained optimization) tBMM is applicable to constrained optimization problems on manifolds. Here, constrained optimization on manifolds means we allow the domain $\Param$ of the optimization problem to be a closed subset of the manifold, i.e. $\Param \subseteq \M$, which is not necessarily the entire manifold.
    \item[(3)] (Robustness) tBMM is robust in the face of inexact computations in each iteration. See \ref{assumption:A0_optimal_gap}\textbf{(ii)}.
\end{description}
tBMM entails various classical and practical algorithms. This includes inexact RGD \eqref{eq:RGD_inexact}, block Riemannian prox-linear updates (see Section \ref{sec:brpl}), and nonsmooth proximal gradient method on Stiefel manifolds (see Section \ref{sec:nonsmooth_st}). We apply our results to the above classical algorithms and obtain the following including empirical findings:
\begin{description}[itemsep=0.01cm]
    \item[(4)] The inexact RGD \eqref{eq:RGD_inexact} for smooth objectives converges to the set of stationary points and has iteration complexity of $O(\eps^{-2})$. See Corollary \ref{cor:inexact_RGD}.
    \item[(5)] We give a convergence and complexity result of $O(\eps^{-2})$ for nonsmooth proximal gradient method on Stiefel manifolds. See Section \ref{sec:nonsmooth_st}.
    \item[(6)] We empirically verified tBMM is faster than the classical algorithm on various problems, including nonnegative tensor decomposition with Riemannian constraints, regularized nonnegative matrix factorization, and low-rank matrix recovery. See Section \ref{sec:stylized_app}.
\end{description}

\subsection{Preliminaries and Notations}

The notations we use in this work are consistent with those of Riemannian optimization literature. In this section, we provide a brief introduction to the notations used in our paper, with further details provided in the Appendix \ref{sec:preliminaries_Riemannian}. We use $T_x \M$ or $T_x$ to denote the tangent space at $x\in\M$ and $\rtr_x$ to denote a retraction at $x$. Retractions provide a way to lift a function $g:\M\rightarrow \R$ onto the tangent space $T_{x}\M$ via its \textit{pullback} $\hat{g} :=g\circ \rtr_{x}:T_{x}\rightarrow \R$. Denote $\rinj(x)$ as the injectivity radius at $x$. For a subset $\Param\subseteq \mathcal{M}$ and $x\in \Param$, define the \textit{lifted constraint set} $T_{x}^{*}$ as
\begin{align}\label{eq:def_lift_constraints}
	\small T^{*}_{x}&:= \left\{ u\in T_{x} \,\left|\,  \begin{matrix}\textup{$\rtr_x(u)=x'$ for some}\\\textup{ $x'\in \Param$ with $d(x,x')\le r_0 /2$}\end{matrix} \right.\right\},
\end{align}
where $r_0$ is the lower bound of the injectivity radius (see \ref{assumption:Optn2}\textbf{(iii)}) and $d(x,x')$ is the geodesic distance between $x$ and $x'$. 

For block Riemannian optimization, we introduce the following notations: For $\param=[\theta^{(1)},\dots,\theta^{(m)}]$, 
	\begin{align}
		\small \grad_{i} f(\param)&:=\left[\begin{matrix}
		    \textup{Riemmanian gradient of}\\\textup{ $\theta \mapsto f(\theta^{(1)},\dots,\theta^{(i-1)}, \theta, \theta^{(i+1)},\dots,\theta^{(m)})$}\end{matrix}\right],\\ 
		\grad f(\param)&:=[ \grad_{1} f(\param),\dots, \grad_{m} f(\param) ].
	\end{align}
 Throughout this paper, we let $(\param_{n})_{n\ge 1}$  denote an output of Algorithm \ref{algorithm:BMM} and write $\param_{n}=[\theta_{n}^{(1)},\dots,\theta_{n}^{(m)}]$ for each $n\ge 1$. For each $n\ge 1$ and $i=1,\dots,m$, denote 
	\begin{align}\label{eq:def_f_n_marginal}
		f_{n}^{(i)}: \theta \mapsto f(\theta_{n}^{(1)},\dots,\theta_{n}^{(i-1)},\theta,\theta_{n-1}^{(i+1)},\dots,\theta_{n-1}^{(m)}),
	\end{align}
	which we will refer to as the $i$th marginal objective function at iteration $n$. 

\vspace{-0.4cm}
	\section{Algorithm}
	\label{sec:Algorithm}
	
	In Algorithm \ref{algorithm:BMM}, we give a precise statement of the tBMM algorithm we stated in high-level at \eqref{eq:tMM}.  We first define majorizing surrogate functions on Riemannian manifolds. 
	\begin{definition}[Tangential surrogates on Riemannian manifolds]
		\label{def:majorizing_surrogates}
		Fix a function $h:\mathcal{M}\rightarrow \R$ and  $\param\in \M$. Denote the \textit{pullback} $\hat{h}:=h\circ \rtr_{\param}:T_{\param}\M\rightarrow \R$. 
		A function $\hat{g}:T_{\param}\M \rightarrow \R$ is a \textit{tangential surrogate} of $h$ at $\param$ if 
  \vspace{-0cm}
		\begin{align}
			\hat{g}(\eta) \ge \hat{h}(\eta) \quad \textup{for all $\eta\in T_{\param}\M$} \quad \textup{and} \quad \hat{g}(\mathbf{0})=\hat{h}(\mathbf{0}). 
		\end{align}
	\end{definition}
	
	\vspace{-0.2cm}
	If $\hat{g}:T_{\param} \M\rightarrow \R$ is a tangential surrogate of $h:\M\rightarrow \R$ and if $\hat{g}$, $h$ are differentiable, then  $\nabla \hat{g}(\mathbf{0})=\hat{h}(\mathbf{0})=\grad h(\param)$. The high-level idea of tBMM is the following. In order to update the $i$th block of the parameter $\theta^{(i)}_{n}$ at iteration $n$, we use first minimize a tangential majorizer of the pullback of the $i$th block objective function, take a step in the resulting direction in the tangent space, and then use retraction to get back to the manifold. We remark that another formulation of Riemannian BMM \cite{li2023convergence} uses majorizers on the manifold without using the tangent spaces. In fact, if we have a majorizer on the manifold, then its pullback is indeed a tangential majorizer. These two formulations of Riemannian BMM coincides on Euclidean spaces but not in general on non-Euclidean manifolds. See the details in Appendix \ref{sec:app_majorizer}.

	\begin{algorithm}
		
		\caption{Tangential Block Majorization-Minimization (tBMM) } 
		\label{algorithm:BMM}
		\begin{algorithmic}[1]
			\STATE \textbf{Input:} $\param_{0}=(\theta_{0}^{(1)},\cdots,\theta_{0}^{(m)})\in \Theta^{(1)}\times \cdots \times \Theta^{(m)}$ (initial estimate); $N$ (number of iterations); $\alpha\in [0,1]$ (step size)  
			
			\FOR{ $n=1,\dots,N$} 
   \STATE Update estimate $\param_{n}=[\theta_{n}^{(1)},\cdots, \theta_{n}^{(m)}]$ by  
			\FOR{ $i=1,\cdots,m$} 
			\STATE {\small $\displaystyle f_{n}^{(i)}(\cdot):= f\left(\theta_{n}^{(1)},\cdots,\theta_{n}^{(i-1)}, \,  \cdot \,\,  ,\theta_{n-1}^{(i+1)},\cdots, \theta_{n-1}^{(m)}\right)$}
            \vspace{-0.3cm}
			 \begin{align}
				&\hspace{-0.5cm}\begin{cases}\label{eq:tMM_MtMt}
					&\hspace{-0.3cm}\hat{g}_{n}^{(i)} \leftarrow \left[ \begin{matrix}\textup{tangential surrogate of $\varphi_{n}^{(i)}$ at $\theta_{n-1}^{(i)}$}\end{matrix} \right] \vspace{0.1cm}\\
                        &\hspace{-0.3cm} G_{n}^{(i)}(\eta):= \hat{g}_{n}^{(i)}(\eta)+\psi_n^{(i)}(\theta_{n-1}^{(i)}+\eta)\vspace{0.2cm}\\
					&\hspace{-0.3cm}V_{n}^{(i)}\in \argmin_{\eta\in T_{\theta_{n-1}^{(i)}}^{*} } G_{n}^{(i)}(\eta) \vspace{-0.1cm}\\
	               &\hspace{-0.3cm}\textup{$\alpha_{n}^{(i)}\leftarrow \alpha$}; \\ 
                    &\hspace{-0.3cm}\textup{(Optionally, $\alpha_{n}^{(i)}\leftarrow$ line search using Alg. \ref{alg:R_line_search} )} \vspace{0.1cm}\\
					&\hspace{-0.3cm}\theta_{n}^{(i)} = \rtr_{\theta_{n-1}^{(i)}}\left(\alpha_{n}^{(i)} V_{n}^{(i)}\right)
				\end{cases}
			\end{align}			
			\ENDFOR
			\ENDFOR
			\STATE \textbf{output:}  $\param_{N}$ 
		\end{algorithmic}
	\end{algorithm}

	Below we give some remarks on the computational aspects of Algorithm \ref{algorithm:BMM}. For Algorithm \ref{algorithm:BMM}, the step size can be simply set as $1$ since the lifted constraint set already restricts the length of tangent vectors (see definition in \eqref{eq:def_lift_constraints}). However, computing the lifted constraint set requires access to specific information about the manifold (see Section \ref{sec:pf_lemma} for details). Alternatively, one can minimize the tangent surrogate over the `tangent cone' (see Appendix \ref{sec:preliminaries_Riemannian}) and apply Riemannian line search in Algorithm \ref{alg:R_line_search} to help determine a step size, although it requires additional computational cost.

    \vspace{-0.2cm}
	\section{Statement of Results}\label{sec:results}

	Here we state our main results concerning the convergence and complexity of our tBMM algorithm (Alg. \ref{algorithm:BMM}) for the constrained block Riemannian optimization problem in \eqref{eq:def_CROPT_block}. 
	
	\vspace{-0.2cm}
	\subsection{Optimality and Complexity Measures} 
	\label{sec:optimality_measures}

	For iterative algorithms, a 
	first-order optimality condition is unlikely to be satisfied exactly in a finite number of iterations, so it is more important to know how the worst-case number of iterations required to achieve an $\eps$-approximate solution scales with the desired precision $\eps$. 
	More precisely, for the multi-block problem \eqref{eq:def_CROPT_block}, we say $\param_{n}=[\theta_{n}^{(1)},\dots,\theta_{n}^{(m)}]\in \Param$ generated by Algorithm \ref{algorithm:BMM} is an \textit{$\eps$-stationary point} of $f=\varphi+\psi$ over $\Param$ if 
 \begin{align}\label{eq:stationary_approximate}
    \sum_{i=1}^{m} \, \left( \begin{matrix}  -  \inf_{u \in T_{\theta_{n}^{(i)}}^{*},\|u\|\le 1 } \, \biggl\langle \grad_{i} \varphi(\param_{n})+ \\\left. \Proj_{T_{\theta_{n}^{(i)}}}\partial \psi_n^{(i)}(\theta^{(i)}_{n}+V_n^{(i)}) ,\, \frac{u}{\hat{r}}  \right\rangle \end{matrix}  \right)\le \eps ,
	\end{align}	
 where $\hat{r}=\min\{r_0,1\}$, i.e. the minimum of $1$ and the lower bound of injectivity radius $r_0$; $V_{n}^{(i)}\in \argmin_{\eta\in T_{\theta_{n-1}^{(i)}}^{*} } \left(G_{n}^{(i)}(\eta):= \hat{g}_{n}^{(i)}(\eta)+\psi_n^{(i)}(\theta_{n-1}^{(i)}+\eta)\right)$. When the objective function is smooth ($\psi=0$), \eqref{eq:stationary_approximate} reduces to the optimality measure for Riemannian constrained optimization problems used in \cite{li2023convergence}. 
    Furthermore, in the Euclidean setting, \eqref{eq:stationary_approximate} reduces to the optimality measure used in \cite{lyu2022convergence,lyu2023block} for constrained nonconvex smooth optimization and is also equivalent to Def.1 in \cite{nesterov2013gradient} for smooth objectives.  
	
	In the unconstrained block-Riemannian setting where  $\Theta^{(i)}=\mathcal{M}^{(i)}$ for $i=1,\dots,m$, the above  \eqref{eq:stationary_approximate} becomes 
	\begin{align}\label{eq:stationary_option2}
	\hspace{-0.5cm}	\sum_{i=1}^{m} \, \lVert \grad_{i} \varphi(\param_{n}) + \Proj_{T_{\theta_{n}^{(i)}}}\partial \psi_n^{(i)}(\theta^{(i)}_{n}+V_n^{(i)}) \rVert \le \eps. 
	\end{align}	
	In the case of single-block $m=1$, the above is consistent with the optimality measure discussed in \cite{chen2020proximal,yang2014optimality}. When the nonsmooth part $\psi=0$, this becomes the standard definition of $\eps$-stationary points for unconstrained Riemannian optimization problems.

	Next, for each $\eps>0$ we define the \textit{(worst-case) iteration complexity} $N_{\eps}$ of an algorithm computing $(\param_{n})_{n\ge 1}$ for solving \eqref{eq:def_CROPT_block} as 
	\begin{align}\label{eq:Neps}
 \small \hspace{-0.5cm}
		N_{\eps}:= \sup_{\param_{0}\in \Param} \inf\, \left\{ n\ge 1 \,\left|\, \begin{matrix}\text{$\param_{n}$ is an $\eps$-approximate} \\ \text{stationary point of $f$ over $\Param$}\end{matrix}\right. \right\}, 
	\end{align}
	where $(\param_{n})_{n\ge 0}$ is a sequence of estimates produced by the algorithm with an initial estimate $\param_{0}$. Note that $N_{\eps}$ gives the \textit{worst-case} bound on the number of iterations for an algorithm to achieve an $\eps$-approximate solution due to the supremum over the initialization $\param_{0}$ in \eqref{eq:Neps}.

	\subsection{Statement of Results}

	In this section, we consider solving the minimization problem \eqref{eq:def_CROPT_block} using Algorithm \ref{algorithm:BMM}, utilizing tangent spaces and retraction for the majorization and minimization steps. One advantage of tBMM is that it utilizes tangent spaces and retractions so that one can bypass handling Riemannian geometry directly. 
	
	We start with stating some general assumptions. We allow inexact computation of the solution to the minimization sub-problems in Algorithm \ref{algorithm:BMM}. This is practical since the minimization of the tangential surrogates may not always be exactly solvable. To be precise, for each $n \geq 1$, we define the \textit{optimality gap} $\Delta_n$ by
	\begin{equation}\label{eq:def_optimality_gap}
		\Delta_n :=
			\max _{1 \leq i \leq m}\bigg(G_n^{(i)}(V_n^{(i)})-\inf _{\eta \in T^{*}_{\theta_{n-1}^{(i)}}} G_n^{(i)}(\eta)\bigg).
	\end{equation}
	For the convergence analysis to hold, we require that the optimal gaps decay fast enough so that they are summable. See Assumption \ref{assumption:A0_optimal_gap}.

	\begin{customassumption}{(A1)}
		\label{assumption:A0_optimal_gap}
		For Alg. \ref{algorithm:BMM}, we make the following assumptions: 
  \vspace{-0.2cm}
		\begin{description}[itemsep=0.01cm]
			\item[(i)] (Objective) The smooth part of objective $\varphi:\Param=\prod_{i=1}^{m} \Theta^{(i)} \rightarrow \R$ is (block-wise) continuously differentiable and the possibly nonsmooth part $\psi$ is convex and Lipschitz continuous with parameter $L_\psi$. The values of objective function $f$ are uniformly lower bounded by some $f^{*}\in \R$. Furthermore, the sublevel sets $f^{-1}((-\infty, a))=$ $\{\boldsymbol{\theta} \in \boldsymbol{\Theta}: f(\boldsymbol{\theta}) \leq a\}$ are compact for each $a \in \mathbb{R}$. 
			
			\item[(ii)] (Inexact computation) The optimality gaps $\Delta_n$ in \eqref{eq:def_optimality_gap} are summable, that is, $\sum_{n=1}^{\infty} \Delta_n<\infty$. 
                \item[(iii)] (Retraction) The retraction $\rtr$ satisfies the following: There exists $M_2 >0$ such that for all $\theta\in \Theta^{(i)}$ and $V\in T_{\theta}\mathcal{M}^{(i)}$,
                \begin{align}
                    \lVert \rtr_{\theta}(V) - (\theta+V) \rVert \le M_{2} \lVert V \rVert^{2} .
                \end{align} 
                
		\end{description}
	\end{customassumption}
 \vspace{-0.2cm}
Note \ref{assumption:A0_optimal_gap}\textbf{(iii)} is a common assumption in the literature of Riemannian optimization, which is used in e.g. \cite{absil2007trust,absil2006convergence,boumal2019global,liu2019quadratic,chen2020proximal}. This assumption is trivially satisfied in Euclidean spaces since the retraction becomes the identity there. Furthermore, \ref{assumption:A0_optimal_gap}\textbf{(iii)} holds when the constrained set $\Theta^{(i)}$ is compact, see Prop. \ref{prop:cpt_manifold_in_Euclidean}.

\vspace{-0.2cm}

\begin{figure}[H]
  \includegraphics[width=.4\linewidth]{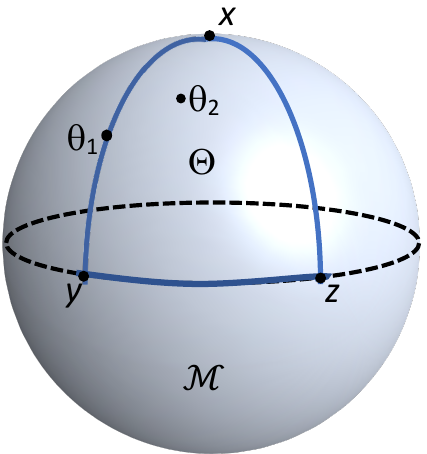}
  \includegraphics[width=.49\linewidth]{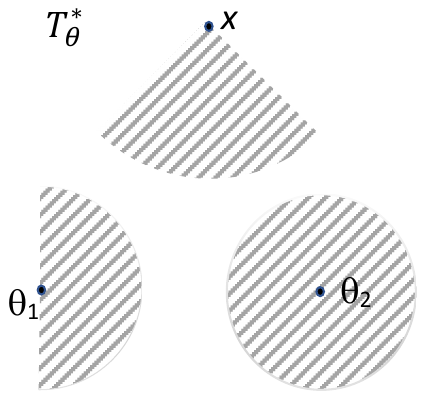}
  \vspace{-0.1cm}
\caption{Illustration of the lifted constraint set $T^{*}_{\theta}$.}
\label{fig:lifted_set}
\end{figure}
		
		\vspace{-0.5cm} 
		For our analysis we impose the following assumption. 
		
		\begin{customassumption}{(A2)}
			\label{assumption:Optn2}
			For Algorithm \ref{algorithm:BMM}, we assume the following: Each tangential surrogate $\hat{g}_{n}^{(i)}:\M^{(i)}\rightarrow \R$ of $\varphi_{n}^{(i)}$ at $\theta_{n-1}^{(i)}$ is differentiable and satisfies the following strong convexity and restricted smoothness properties: 
   \vspace{-0.2cm}
			\begin{description}[itemsep=-0.5cm]				\item[(i)] (Strong convexity) There exists a constant $\rho_{n}^{(i)}>0$ such that for all $\eta,\xi\in T_{\param}$,
    \vspace{-0.5cm}
				\begin{align}
					\hat{g}_{n}^{(i)}(\xi+\eta)  \ge \hat{g}_{n}^{(i)}(\xi) + \langle \nabla \hat{g}_{n}^{(i)}(\xi) ,\, \eta  \rangle - \frac{\rho_{n}^{(i)}}{2}\lVert \eta \rVert^{2}. 
				\end{align}

				\item[(ii)] (Restricted smoothness) There exists a constant $L_{n}^{(i)}>0$ such that for all $\eta\in T_{\param}$,
			\vspace{-0.5cm}	\begin{align}\label{eq:g_surrogate_g_smooth}
					| \hat{g}_{n}^{(i)}(\eta) - \hat{g}_{n}^{(i)}(\mathbf{0}) - \langle \nabla \hat{g}_{n}^{(i)}(\mathbf{0})  ,\, \eta \rangle| \le \frac{L_{n}^{(i)}}{2} \lVert \eta \rVert^{2}.
				\end{align} 
			\end{description}\vspace{-0.7cm}
			\begin{description}
				\item[(iii)] Each $\Theta^{(i)}\subseteq \mathcal{M}^{(i)}$ 
				is (geodesically) strongly convex and there exists a uniform lower bound $r_{0}>0$ for $\rinj(x)$ over $x\in \Theta^{(i)}$. Moreover, the lifted constraint set $T^*_{\param^{(i)}}$ (see \eqref{eq:def_lift_constraints}) is convex for each $\theta^{(i)}\in \Theta^{(i)}$.
			\end{description}
		\end{customassumption}

\vspace{-0.5cm} 
The assumption of convexity of the lifted constraint set $T^*_{\param^{(i)}}$ in \ref{assumption:Optn2}\textbf{(iii)} makes the sub-problem of finding $V_n^{(i)}$ in \eqref{eq:tMM_MtMt} a convex problem, which is easy to solve using classical methods \cite{boyd2004convex,nesterov1998introductory,nesterov2018lectures}. For Euclidean optimization problems, i.e. $\M^{(i)}=\R^{I_n}$ with $I_n\in \mathbb{Z}_+$, the tangent spaces are identical to $\R^{I_n}$. Hence $T^*_{\param^{(i)}}=\Theta^{(i)}$, whose convexity is already guaranteed by the first part of \ref{assumption:Optn2}\textbf{(iii)}. More generally, for unconstrained optimization problems on Hadamard manifolds (see Section \ref{sec:examples_manifold}) which include Euclidean spaces, hyperbolic spaces, manifolds of positive definite matrices, when taking $\Exp$ as the retraction in \eqref{eq:def_lift_constraints}, we have $T^*_{\param^{(i)}}=T_{\param^{(i)}}$ since the exponential map is a global diffeomorphism on Hadamard manifolds (see Theorem 4.1, p.221 in \cite{sakai1996riemannian}). Hence the convexity of $T^*_{\param^{(i)}}$ is guaranteed without further assumptions in this case. In fact, for the unconstrained problem on any Riemannian manifolds, when taking $\Exp$ as the retraction in \eqref{eq:def_lift_constraints}, which is the same setting as in \cite{gutman2023coordinate}, $T^*_{\param^{(i)}}$ becomes a ball of radius $r_0$ centered at $\param^{(i)}$ so the convexity is also guaranteed. In Figure \ref{fig:lifted_set}, an illustration of $T^*_{\param}$ that satisfies \ref{assumption:Optn2}\textbf{(iii)} is shown. Additional details about this concrete example can be found in Appendix \ref{sec:lifted_set}.

		Our main result for Algorithm \ref{algorithm:BMM} is stated in Theorem \ref{thm:RBMM_MtMt}.

		\begin{theorem}[Convergence of tBMM]\label{thm:RBMM_MtMt}
			Let $f$ denote the objective function in \eqref{eq:def_CROPT_block} with $m\ge 1$. 
			Let $(\param_{n})_{n\ge 0}$ be an output of Algorithm \ref{algorithm:BMM}. Suppose \ref{assumption:A0_optimal_gap} and \ref{assumption:Optn2} hold. Let $\hat{r}=\min\{1,r_0\}$.
			Then the following hold. 
   \vspace{-0.3cm}
			\begin{description}[itemsep=-0.1cm]
				\item[(i)] Suppose there exist constants $\rho, L>0$ such that for all $n\ge 1$ and $i=1,\dots, m$,  $\rho_{n}^{(i)}\ge \rho$, $L_{n}^{(i)}\le L$, and $\alpha_{n}^{(i)}\in \left(0 ,\, \alpha  \right]$ with $\alpha = \rho/(\rho+2L_\psi M_2)$. Let $M=f(\param_{0}) - f^{*} + m\sum_{n=1}^{\infty} \Delta_{n}$. Then for $N\ge \frac{L^2}{(L+C')^2}\frac{2M}{\rho\alpha}$, 

    {\small
		\vspace{-0.3cm}	\begin{align}\label{eq:tMM_MtMt_complexity1}
            &\begin{matrix}
			\displaystyle\hspace{-1cm}    \min_{1\le k \le N} \sum_{i=1}^{m}   - \inf_{\eta \in T^*_{\theta^{(i)}_{k-1}},\|\eta\|\le 1} \biggl \langle \grad \varphi_k^{(i)}(\theta^{(i)}_{k-1})+\\
       \displaystyle\hspace{0.5cm} \left.\Proj_{T_{\theta^{(i)}_{k-1}}} \partial \psi_k^{(i)}(\theta^{(i)}_{k-1}+V_k^{(i)}),\frac{\eta}{\min\{r_0,1\}}\right\rangle \end{matrix} 
       \end{align}
       \vspace{-0.5cm}
       \begin{align}
       \le \frac{L +C}{\min\{r_0,1\}}\sqrt{\frac{2M}{ \rho \alpha N}}.
		\end{align}	
	}
				\item[(ii)] In addition to the hypothesis in \textbf{\textup{(i)}}, further assume that $\psi=0$. 
    Then every limit point of $(\param_{n})_{n\ge 0}$ is a stationary point of \eqref{eq:def_CROPT_block}. 
				
				\item[(iii)] In addition to the hypothesis in \textbf{\textup{(i)}}, further assume  $\varphi$ is joint smooth (Def. \ref{def:joint_smooth}). Let $M'=(4\alpha+6) m\sum_{n=1}^{\infty}\Delta_n + 6(f(\param_{0}) - f^{*}) $. Then for each $N\ge \frac{L^2}{(L+C'+L')^2}\frac{M^2}{M'\rho\alpha}$, 
				\begin{align}
    \label{eq:tMM_MtMt_complexity2}
			\vspace{-0.3cm}&\begin{matrix}\displaystyle \hspace{-1cm}\min_{1\le k \le N} \sum_{i=1}^{m}   - \inf_{\eta \in T^*_{\theta^{(i)}_{k-1}},\|\eta\|\le 1}\biggl\langle \grad_i \varphi(\param_{k-1}) +\\ \displaystyle\hspace{0.7cm}\left.\Proj_{T_{\theta^{(i)}_{k-1}}} \partial \psi_k^{(i)}(\theta^{(i)}_{k-1}+V_k^{(i)}),\frac{\eta}{\min\{r_0,1\}}\right\rangle \end{matrix}\\
   &\hspace{0cm}\le \frac{2(L+C+ L')}{\min\{r_0,1\}} \sqrt{\frac{M'}{\rho \alpha N}},
		\end{align}
  In particular, the worst case iteration complexity is $N_{\eps}=O(\eps^{-2})$. 
			\end{description}
		\end{theorem}
	\vspace{-0.1cm}	
		Note that \eqref{eq:tMM_MtMt_complexity1} in Theorem \ref{thm:RBMM_MtMt} \textbf{(i)} gives a type of iteration complexity result but the measure of optimality (i.e., the left-hand side of \eqref{eq:tMM_MtMt_complexity1}) is adapted to cyclic block optimization algorithms. In order to relate this to the more generic optimality measure in \eqref{eq:stationary_option2}, we need to bound the difference between $\grad \varphi_{k}^{(i)}(\theta_{k-1}^{(i)})$ and $\grad_{i} \varphi(\param_{k-1})$. This can be achieved by the additional smoothness property of $\varphi$ as in \eqref{eq:thm_MtMt_smoothness} and the resulting iteration complexity result is stated in \eqref{eq:tMM_MtMt_complexity2} in Theorem \ref{thm:RBMM_MtMt} \textbf{(iii)}. We put the proof of Theorem \ref{thm:RBMM_MtMt} along with some key lemmas in Appendix \ref{sec:pf_lemma}.

	\section{Applications}
	\label{sec:examples}

	In this section, we discuss some examples of tBMM and its connection to some other classical algorithms.

 \subsection{Example Manifolds}\label{sec:examples_manifold}

    In this section, we list several examples of manifolds that appear in various machine-learning problems. 

    \begin{example}[Stiefel manifold]
		\label{eg:stiefel}
		\normalfont 
		
		The Stiefel manifold $\mathcal{V}^{n\times k}$ is the set of all orthonormal $k$-frames in $\mathbb{R}^n$. Namely, it is the set of all  ordered orthonormal $k$-tuples of vectors in $\R^{n}$, i.e. $\mathcal{V}^{n\times k}=\left\{A \in \mathbb{R}^{n \times k}: A^T A=I_k\right\}$, 
		where $I_k$ is the $k\times k $ identity matrix. For $X\in \R^{n\times k}$, let $X=U\Sigma V^{T}$ be the SVD, then the projection of $X$ onto $\mathcal{V}^{n\times k}$ is given by $\textup{Proj}_{\mathcal{V}^{n\times k}}(X)=UV^{T}$.
		
	\end{example}

    \begin{example}[Manifold of fixed-rank matrices]
		\label{eg:fixed-rank}
		\normalfont
		The set $\mathcal{R}_r=\left\{X \in \mathbb{R}^{n \times m}: \operatorname{rank}(X)=r\right\}$ is a smooth submanifold of $\R^{n \times m}$. For $X\in \R^{n\times m}$, let $X=U\Sigma V^{T}$ be the SVD, and let $u_i$ and $v_i$ be the $i$-th column of $U$ and $V$, respectively. Write the singular values in $\Sigma$ in nonincreasing order, $\sigma_1(X) \geq \sigma_2(X) \geq \cdots \geq \sigma_{\min \{n, m\}}(X) \geq 0$. Then the projection of $X$ onto $\mathcal{R}_r$ is given by
		\begin{equation}
			\textup{Proj}_{\mathcal{R}_r}(X)= \sum_{i=1}^{r}\sigma_{i}(X)u_i v_i^{T} := U\Sigma_{r} V^{T}.
		\end{equation}
	\end{example}

    \vspace{-0.2cm}
    Another class of manifolds that is widely studied in the literature is the Hadamard manifolds, which includes many commonly encountered manifolds. \textit{Hadamard manifolds} are complete and simply connected Riemannian
	manifolds with nonpositive sectional curvature ( \cite{Burago2001ACI} and \cite{burago1992alexandrov}). The injectivity radius at every point is infinity on a Hadamard manifold. Examples of Hadamard manifolds can be found in Appendix \ref{sec:Had_eg}. 

		\subsection{Block Riemannian Prox-linear Updates and Block Riemannian GD}
		\label{sec:brpl}

		A primary approach in tBMM is to use the Riemannian prox-linear surrogates. We begin our discussion from the single-block case for \eqref{eq:def_CROPT} without constraints (i.e., $\Param=\M$) for simplicity. Given the iterate $\param_{n-1}\in \M$, the Riemannian prox-linear surrogate $\hat{g}:T_{\param_{n-1}}\M\rightarrow \R$  takes the form 
		\begin{align}\label{eq:Riemannian_prox_linear_single}
			\hspace{-0.5cm}\hat{g}_{n}(\eta) := f(\param_{n-1}) + \langle \grad f (\param_{n-1}),\, \eta \rangle + \frac{\lambda}{2} \lVert \eta\rVert^{2},
		\end{align}
		where $\lambda>0$ is the proximal regularization parameter. Note that $\hat{g}$ is defined only on the tangent space at $\param_{n-1}$ and not on the manifold. In order for $\hat{g}_{n}$ to be a tangential surrogate of $f$ at $\param_{n-1}$, one can impose the following restricted smoothness property for the pullback objective $\hat{f}:=f\circ \rtr_{\param_{n-1}}$: For some constant $L>0$, 
		\begin{align}\label{eq:f_Lipschitz_type_grad}
			\left| \hat{f}(\eta) - \hat{f}(\mathbf{0}) - \langle \grad f(\param_{n-1}),\, \eta \rangle \right| \le \frac{L}{2} \lVert \eta \rVert^{2}. 
		\end{align}
		The above property was introduced in \cite{boumal2019global} under the name of `restricted Lipschitz-type gradients', and was critically used to analyze Riemannian gradient descent and Riemannian trust region methods. Hence under \eqref{eq:f_Lipschitz_type_grad} and assuming $\lambda\ge L$, $\hat{g}_{n}$ in \eqref{eq:Riemannian_prox_linear_single} is indeed a tangential majorizer of $f$ at $\param_{n-1}$. The resulting tBMM with the Riemannian prox-linear surrogate above reads as 
		\begin{align}
			\param_{n}\leftarrow \rtr_{\param_{n-1}}\left( -\frac{1}{\lambda}\grad f(\param_{n-1}) \right), 
		\end{align}
		which is the standard Riemannian gradient descent with a fixed step size $1/\lambda \le 1/L$. Our convergent result Theorem \ref{thm:RBMM_MtMt} holds for the above updates whenever $\lambda\ge L$. 
		
		Next, when we further impose an additional strongly convex constraint $\Param\subseteq \M$, the decent direction is not necessarily the negative Riemannian gradient $-\grad f(\param_{n-1})$ and one needs to solve a quadratic minimization over the tangent space. The portion of tBMM in this case reads as 
		\begin{align}\label{eq:tBMM_rpl}
        \hspace{-0.5cm}
			\begin{cases}
				V_{n}\leftarrow   \argmin_{\eta\in T^*_{\param_{n-1}}} \left\lVert \eta - \left( -\frac{1}{\lambda} \grad f(\param_{n-1})\right)  \right\rVert^{2}\\ 
\hspace{1cm} = \textup{Proj}_{T^*_{\param_{n-1}}} \left( -\frac{1}{\lambda} \grad f(\param_{n-1})\right) , \\
				\param_{n}\leftarrow \rtr_{\param_{n-1}}\left( \alpha_n V_{n} \right). 
			\end{cases}
		\end{align}
where $\alpha_{n}$ is the step size that could be either a small enough constant or decided by Algorithm \ref{alg:R_line_search}. Namely, the constrained descent direction $V_{n}$ is the projection of the unconstrained descent direction $-\frac{1}{\lambda} \grad f(\param_{n-1})$ onto the lifted constraint set $T^*_{\param_{n-1}}$. The above is a Riemannian version of the projected gradient descent that our tBMM algorithm entails. As in the unconstrained case, our convergent result Theorem \ref{thm:RBMM_MtMt} holds for the above updates whenever $\lambda\ge L$. 
		
		
		We can apply a similar approach for the constrained block Riemannian optimization problem \eqref{eq:def_CROPT_block}. Namely, for each iteration $n$ and block $i$, use the following Riemannian prox-linear surrogate
		\begin{align}\label{eq:rpl_multi}
			\hat{g}_{n}^{(i)}(\eta):= f_{n}^{(i)}(\theta_{n-1}^{(i)} ) + \langle \grad f_{n}^{(i)}(\theta_{n-1}^{(i)}) ,\,  \eta  \rangle + \frac{\lambda}{2} \lVert \eta \rVert^{2}. 
		\end{align}
		Then the resulting portion of tBMM reads as the following block projected Riemannian gradient descent: For $i=1,\dots,m$,
		\begin{align}
			&\begin{cases}
				&V_{n}^{(i)}\leftarrow   \textup{Proj}_{T^{*}_{\theta_{n-1}^{(i)}}} \left( -\frac{1}{\lambda} \grad f(\theta_{n-1}^{(i)})\right), \\[5pt]
				&\,\,\theta_{n}^{(i)}\leftarrow \rtr_{\theta_{n-1}^{(i)}}( \alpha_{n}^{(i)} V_{n}^{(i)} ),
			\end{cases}
		\end{align}
		where $\alpha_{n}^{(i)}$ is the step size that could be either a small enough constant or decided by Algorithm \ref{alg:R_line_search}. 

  It is worth mentioning that utilizing the Riemannian gradient ``\textup{grad}'' instead of the full Euclidean gradient ``$\nabla$'' can provide significant computational savings.  One extensively investigated example is RPGD on fixed-rank manifolds (Example \ref{eg:fixed-rank}). There, computing $\textup{Proj}_{\mathcal{R}_r} \left( \theta^{(i)}_{n-1} - \frac{1}{\lambda} \nabla \right)$ involves computing the full SVD of a $n\times m$ matrix, while $\textup{Proj}_{\mathcal{R}_r} \left( \theta^{(i)}_{n-1} - \frac{1}{\lambda} \grad \right)$ only involves computing the SVD of a much smaller $2r \times 2r$ matrix (see e.g. \cite{wei2016guarantees}) when $r\ll \min\{m,n\}$.

  \subsection{Nonsmooth Proximal Gradient Method on the Stiefel Manifold \cite{chen2020proximal}}\label{sec:nonsmooth_st}
	In \cite{chen2020proximal}, the authors focus on the nonsmooth optimization problem on the Stiefel manifold as follows,
\begin{equation}\label{eq:nonsmooth_st}
    \min_{\param \in \mathcal{V}^{n\times p}} f(\param) := \phi(\param)+\psi(\param),
\end{equation}
where $\phi$ is Euclidean smooth with parameter $L_\phi$, $\psi$ is convex and Lipschitz continuous with parameter $L_\psi$ but is not necessarily Euclidean smooth. In the paper, the authors propose the following algorithm to solve \eqref{eq:nonsmooth_st},
\begin{align}
    \begin{cases}
    G(V):= \phi(\param_{n-1}) + \langle \grad \phi(\param_{n-1}) , V\rangle  \\
    + \lambda \|V\|^2_F + \psi(\param_{n-1} + V) \\
        V_n = \argmin_{V\in T_{\param_{n-1}}} G(V)\\
        \param_n = \rtr_{\param_{n-1}}(\alpha V_n), \label{eq:st_updates}
    \end{cases}
\end{align}
 where $\alpha>0$ is a step size. Recall the Stiefel manifold is compact, in order to compare and bound the difference of vectors before and after retraction, the authors in \cite{chen2020proximal} use the following proposition substantially to establish the complexity result.

\begin{prop}[Compact manifold embedded in Euclidean space \cite{boumal2019global,liu2019quadratic}]
    \label{prop:cpt_manifold_in_Euclidean}
    Let $\mathcal{M}$ be a compact Riemannian manifold embedded in an Euclidean space with norm $\lVert \cdot \rVert$. Then there exist constants $M_{1},M_{2}>0$ such that for all $\theta\in \mathcal{M}^{(i)}$ and $V\in T_{\theta}\mathcal{M}$,
     \begin{align}
        &\lVert \rtr_{\theta}(V) - \theta \rVert \le M_{1} \lVert V \rVert, \\
        & \lVert \rtr_{\theta}(V) - (\theta+V) \rVert \le M_{2} \lVert V \rVert^{2}.
    \end{align}
\end{prop}

In fact, when replacing the compact manifold with a compact constrained set on an arbitrary embedded submanifold, the bounds in Proposition \ref{prop:cpt_manifold_in_Euclidean} still hold. This is one situation that our assumption in \ref{assumption:A0_optimal_gap}\textbf{(iii)} holds. When $\phi$ in \eqref{eq:nonsmooth_st} is $L_\phi$-smooth, the algorithm \eqref{eq:st_updates} with proper regularization parameter $\lambda$ falls into our tBMM framework, and therefore we can apply the convergence and complexity results in Theorem \ref{thm:RBMM_MtMt}. The complexity results are formally stated in the following corollary. The proof is in Appendix \ref{sec:proof_eg}.

\begin{corollary}[Complexity of nonsmooth proximal gradient method on Stiefel manifolds]\label{cor:nonsmooth_Stiefel}
Let $\phi$ in \eqref{eq:nonsmooth_st} be $L_\phi$-smooth. Then the complexity results in Theorem \ref{thm:RBMM_MtMt} hold for the algorithm \eqref{eq:st_updates} when $\lambda\ge L_\phi$. In particular, the algorithm \eqref{eq:st_updates} has iteration complexity of $O(\eps^{-2})$. 
\end{corollary}

We remark that in \cite{chen2020proximal}, the authors established complexity results based on a different definition of $\eps$-stationary point, namely, $\param_n$ is $\eps$-stationary point if $\|V_n\|\le \eps/ L_\phi$. Though this definition brings practical benefits, it is not a generic definition. The definition of $\eps$-stationary point \eqref{eq:stationary_approximate} we used is more generic, and is a natural generalization of the corresponding concept in the Euclidean setting. Moreover, our definition allows further constraints on the manifold, i.e. the constraint set $\Param$ is not necessarily the entire manifold $\mathcal{M}$.



\section{Numerical Validation}	\label{sec:stylized_app}
In this section, we provide numerical validation of tBMM on various problems. Additional details of the experiments are provided in Appendix \ref{sec:appendix_numerics}.
\begin{figure}
  \includegraphics[width=.49\linewidth]{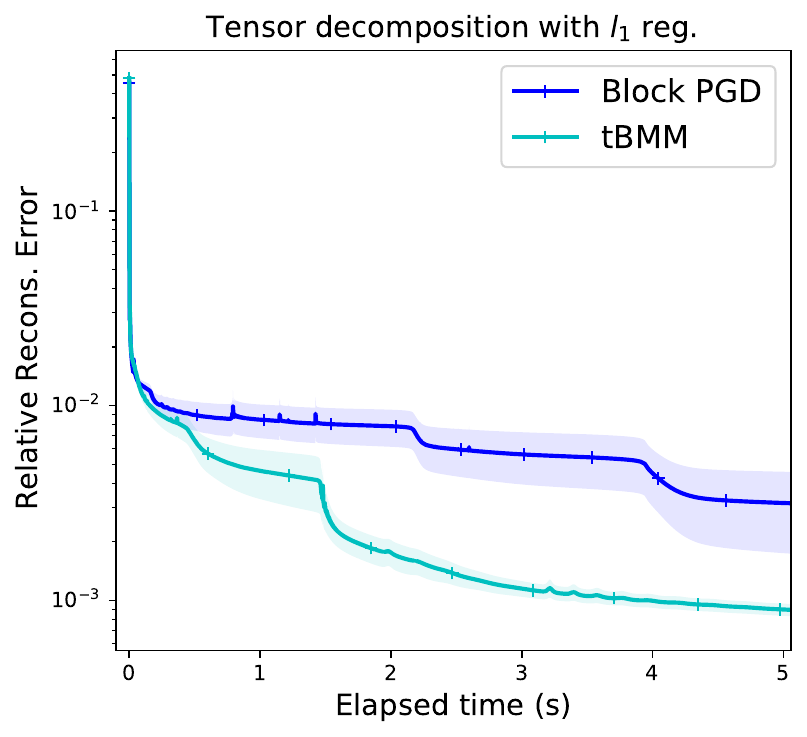}
  \includegraphics[width=.49\linewidth]{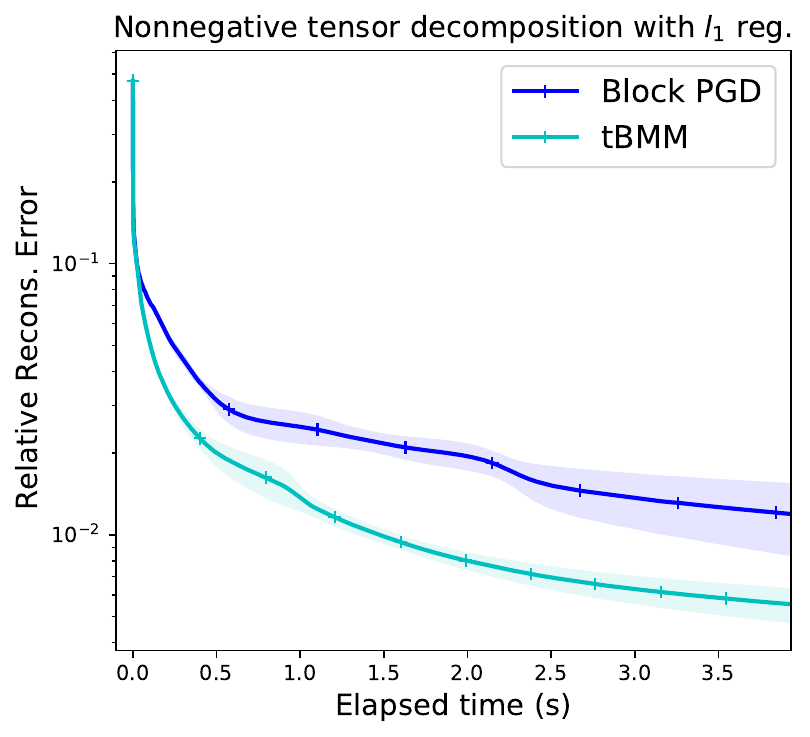}%
  \vspace{-0.1cm}
\caption{Comparison of tBMM and block projected gradient descent on (nonnegative) tensor decomposition problems with $l_1$ regularizations. Relative errors with standard deviation are shown by the solid lines and shaded regions of respective colors.}
\label{fig:NCPD}
\end{figure}

\textbf{Nonnegative Tensor Decomposition with Riemannian Constraints.} Given a tensor $\X \in  \R^{I_{1}\times \cdots \times I_{m}}$ and a fixed integer $R\ge 1$, the tensor decomposition problem aims to find the loading matrices $U^{(i)} \in \R^{I_{i}\times R}$ for $i=1,\dots, m$ such that the sum of the outer products of their respective columns approximate $\X$: $\X \approx  \sum_{k=1}^{R} \bigotimes_{i=1}^{m} U^{(i)}[:,k]$, where $U^{(i)}[:,k]$ denotes the $k^{\textup{th}}$ column of the $I_{i}\times R$ loading matrix $U^{(i)}$ and $\bigotimes$ denotes the outer product. One can also impose Riemannian constraints on the loading matrices, such that each $U^{(i)} \in \mathcal{M}^{(i)}$ resides on a Riemannian manifold. This Riemannian tensor decomposition problem can be formulated as the following optimization problem:
\begin{equation}  
\label{eq:NTF} 
\small
\argmin_{U^{(i)}\in \Theta^{(i)}} \left( \begin{matrix}  f(U^{(1)},\dots, U^{(m)}) := \\ \left\lVert \X - \sum_{k=1}^{R} \bigotimes_{i=1}^{m} U^{(i)}[:,k] \right\rVert_{F}^{2} \hspace{-0.7em}+ \sum_{i=1}^{m}\lambda_{i} \lVert U^{(i)} \rVert_{1} \end{matrix}\right)\hspace{-0.3em}, 
\end{equation}
where $\Theta^{(i)}\subseteq \mathcal{M}^{(i)}$ denotes a closed and geodesically convex constraint set and $\lambda_{i}\ge 0$ is the parameter of  $\ell_{1}$-regularizer. 

In our experiments as shown in Figure \ref{fig:NCPD}, we let the first block $\mathcal{M}^{(1)}$ be the fixed-rank manifold $\mathcal{R}_r$ and the other blocks be Euclidean spaces. We consider two different cases: with or without nonnegativity constraints for each loading matrix. In both cases, we use the prox-linear surrogates (see Sec. \ref{sec:nonsmooth_st}) for tBMM. Due to the $\ell_1$ regularizer in \eqref{eq:NTF}, the subproblem of minimizing the tangential surrogate $\hat{G}$ does not have a closed-form solution. Furthermore, when nonnegativity is imposed, another projection onto the positive orthant is required when solving the subproblems. Hence, one needs to implement an iterative algorithm for solving the subproblems. This brings inexactness to the solution to the subproblems. In Figure \ref{fig:NCPD}, the left figure is the reconstruction error plot of solving problem \eqref{eq:NTF} without nonnegativity constraints, and the right figure is that with nonnegativity constraints. In both cases, tBMM outperforms block projected gradient descent (block PGD).

\begin{figure}
  \includegraphics[width=.49\linewidth]{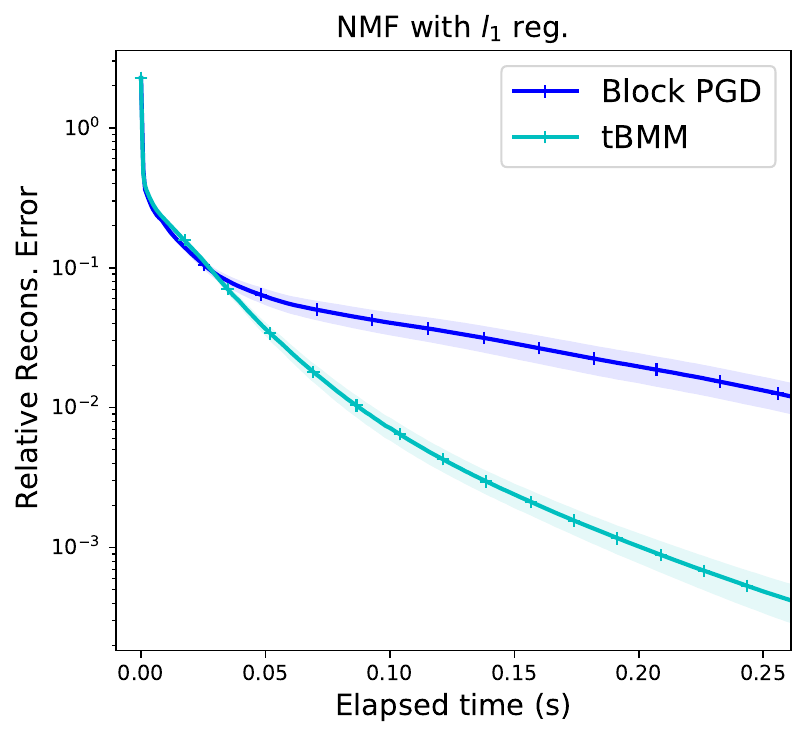}
  \includegraphics[width=.49\linewidth]{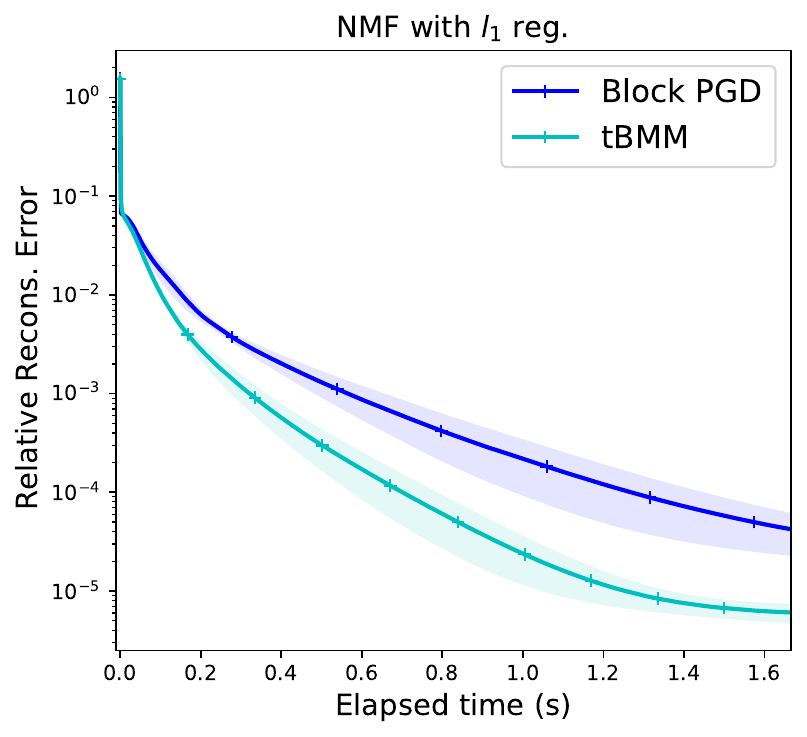}%
  \vspace{-0.1cm}
\caption{Comparison of tBMM and block projected gradient descent on nonnegative matrix factorization problems with $l_1$ regularizations. Relative errors with standard deviation are shown by the solid lines and shaded regions of respective colors.}
\label{fig:NMF}
\end{figure}

\textbf{Regularized Nonnegative Matrix Factorization with Riemannian Constraints.} Given a data matrix $X\in \R^{p\times N}$, the constrained matrix factorization problem is formulated as the following,
\begin{align*}
		\argmin_{W\in \mathcal{M}^{(1)}\subseteq \R^{p\times R}, \, H\in \mathcal{M}^{(2)}\subseteq \R^{R\times N}} \lVert X - WH \lVert_{F}^{2} + \lambda \lVert H \rVert_{1}. 
	\end{align*}
We consider a similar setting as the tensor decomposition problem in the last section. Namely, the first block $\mathcal{M}^{(1)}$ is a fixed-rank manifold $\mathcal{R}_r$ with nonnegativity constraint and the second block is Euclidean space. Similarly, the nonnegativity constraints and the $\ell_1$ regularizer make the subproblems have no closed-form solution. Therefore, the iterative algorithms for solving it bring inexactness to the solution to the subproblem. In Figure \ref{fig:NMF}, we showed the error plot with slightly different settings. In the left figure, there is no nonnegativity constraint to the second block $H$. In the right figure, we impose a nonnegativity constraint on both blocks. In both cases, tBMM using prox-linear surrogates outperforms block PGD.

\textbf{Low-rank Matrix Recovery.} Consider the \textit{low-rank matrix recovery} problem \cite{donoho2006compressed,candes2009accurate} formulated as the following in \cite{tanner2013normalized,wei2016guarantees},
\begin{align}\label{eq:compressed_sensing}
    \min_{X\in \mathcal{R}_r\subseteq \R^{m\times n}} \left( f(X):= \frac{1}{2}\|\mathcal{A}(X)-b\|^2\right)
\end{align}
where $\mathcal{R}_r$ is the manifold of fixed-rank matrices, see Example \ref{eg:fixed-rank}. The linear map $\mathcal{A}:\R^{m\times n}\to \R^p$ maps an $m\times n$ matrix to a $p$ dimensional vector, given by
\begin{align}
    \mathcal{A}(X)_i := \langle A_i, X\rangle, \;\; \textup{for}\;\; i=1,\cdots,p.
\end{align}
The $p$-dimensional vector $b$ represents the observations. Given the observations from the sensing operator $\mathcal{A}$, low-rank matrix recovery aims to find the fixed-rank matrix $X$. We define the umdersampling ratio as $\rho = p/mn$. 

We compare the performance of the proposed tBMM (Algorithm \ref{algorithm:BMM}) applied to the low-rank matrix recovery problem against \textit{normalized iterative hard thresholding} (NIHT) \cite{tanner2013normalized}, see details in Appendix \ref{sec:appendix_matrix_rec}. To solve \eqref{eq:compressed_sensing}, we use the prox-linear surrogates \eqref{eq:rpl_multi} for tBMM (Algorithm \ref{algorithm:BMM}) as discussed in Section \ref{sec:brpl}. The resulting updates for this single block problem are shown in \eqref{eq:tBMM_rpl}. 

\begin{figure}
  \includegraphics[width=.49\linewidth]{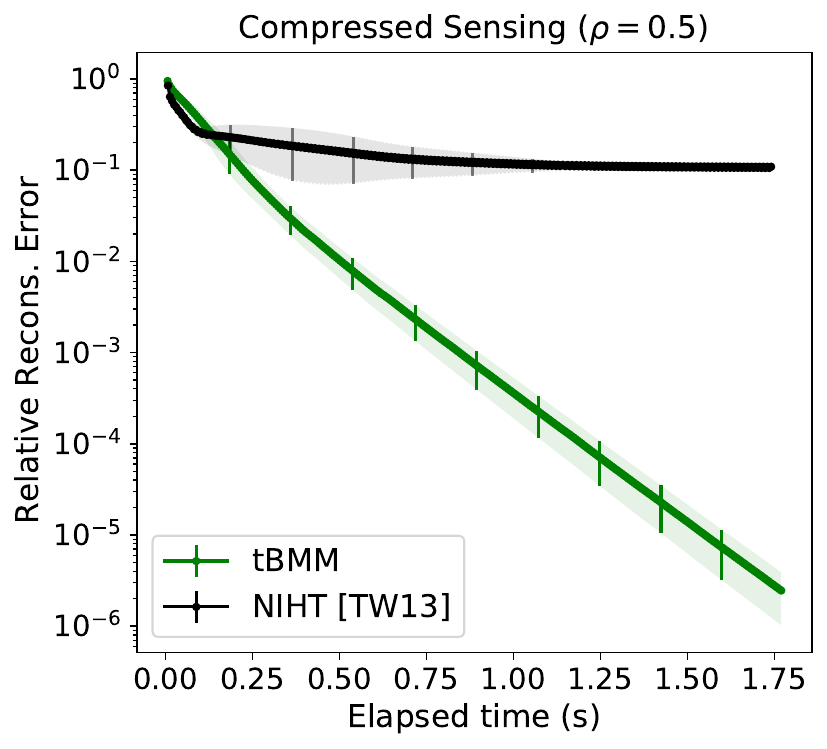}
  \includegraphics[width=.49\linewidth]{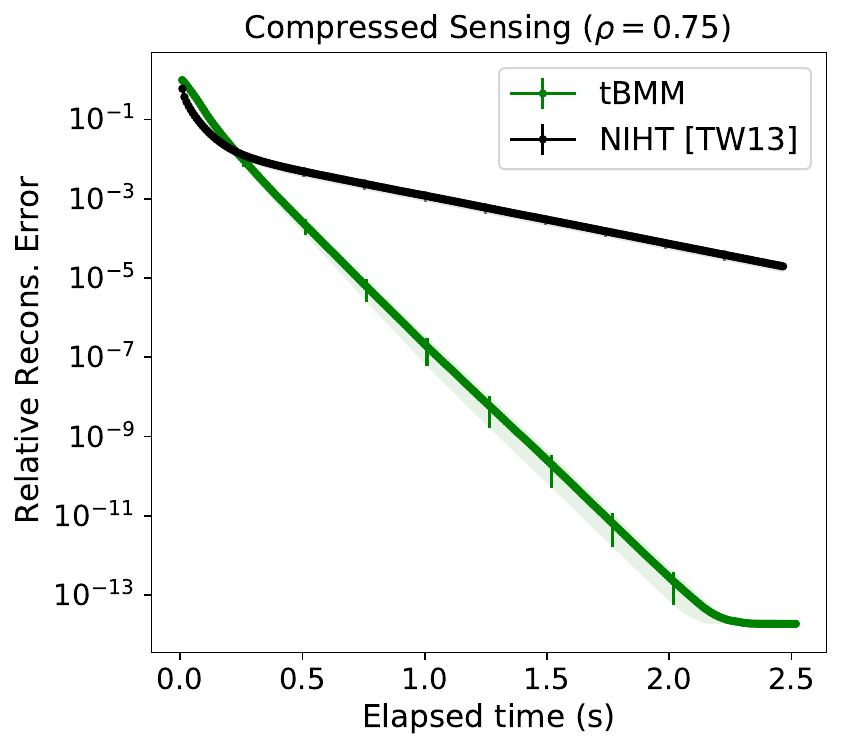}%
  \vspace{-0.4cm}
\caption{Comparison of tBMM and NIHT on the problem of low-rank matrix recovery. Relative errors with standard deviation are shown by the solid lines and shaded regions of respective colors.}
\label{fig:compressed_sensing}
\end{figure}

In Figure \ref{fig:compressed_sensing}, we compare the performance of tBMM and NIHT for $\rho=0.5$ and $0.75$ respectively. Under both settings, tBMM outperforms NIHT in terms of relative reconstruction error. Moreover, tBMM is observed to be more robust for different dimensions of the matrices. The better performance of tBMM over NIHT can be attributed to two main factors: Firstly, as discussed in Section \ref{sec:brpl}, NIHT needs to solve a full SVD of a matrix of dimension $m\times n$ in each iteration, while tBMM only needs to solve that of a $2r \times 2r$ matrix, which provides computational savings. Secondly, the performance of NIHT is affected by the \textit{oversampling ratio} defined as $\mu = r(m+n-r)/p$. When $\mu >1/2$, the recovery of all low-rank matrices by the algorithm is impossible \cite{tanner2013normalized,eldar2011unicity}, which leads to the relatively poor performance of NIHT in the left panel of Figure \ref{fig:compressed_sensing}.

\textbf{Inexact RGD.} In the introduction, we state that the convergence and complexity of inexact RGD can be deduced as a corollary of our Theorem 3.1. Here, we provide numerical validations based on the low-rank matrix recovery problem. In Figure \ref{fig:inexact_RGD}, we verify the performance of inexact tBMM and inexact RGD. We pose inexactness by adding a noise term $\Delta_n = c/(n+1)^2$ to the Riemannian gradient, where $n$ is the iteration number. The $\Delta_n$s are summable, which satisfies \ref{assumption:A0_optimal_gap}\textbf{(ii)}. The inexact tBMM and inexact RGD show compelling convergence speed.

\begin{figure}
  \includegraphics[width=.49\linewidth]{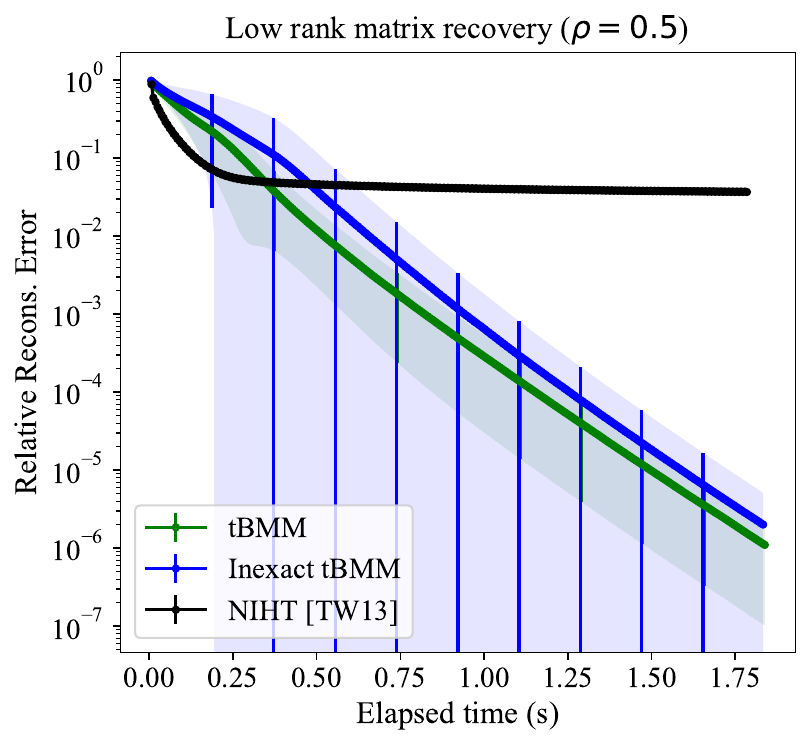}
  \includegraphics[width=.49\linewidth]{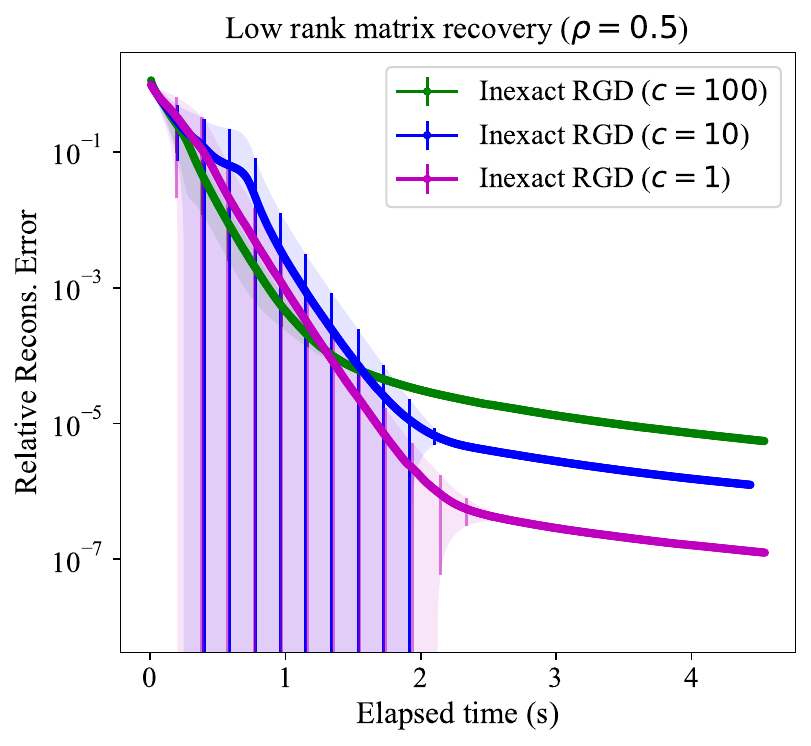}%
  \vspace{-0.4cm}
\caption{Illustration of the convergence of inexact RGD on the low-rank matrix recovery problem. Relative errors with standard deviation are shown by the solid lines and shaded regions. 
}
\label{fig:inexact_RGD}
\end{figure}

\section{Conclusion and Limitations}
In this paper, we proposed a general framework named tBMM 
which entails many classical first-order Riemannian optimization algorithms, including the inexact RGD and the proximal gradient method on Stiefel manifolds. tBMM is applicable to solving multi-block Riemannian optimization problems, and can handle additional (geodesically convex) constraints within each manifold, as well as nonsmooth nonconvex objectives. We established the convergence and complexity results of tBMM. Namely, tBMM converges to a $\eps$-stationary point within $O(\eps^{-2})$ iterations. The convergence and complexity results still hold when the subproblem in each iteration is computed inexactly, as long as the optimality gap is summable. We validated our theoretical results on tBMM through various numerical experiments. We also demonstrated that tBMM can show improved performance on various Riemannian optimization problems compared to existing methods.

\newpage

	\section*{Impact Statements}

  This paper presents work whose goal is to advance the field of Machine Learning. There are many potential societal consequences of our work, none of which we feel must be specifically highlighted here.

\section*{Acknowledgements}
We thank all the reviewers for the thoughtful comments and suggestions. YL is partially supported by the Institute for Foundations of Data Science RA fund through NSF Award DMS-2023239 and by the National Science Foundation through grants DMS-2206296. HL is partially supported by the National Science Foundation through grants DMS-2206296 and DMS-2010035. DN is partially supported by NSF DMS-2011140. LB is partially supported by NSF CAREER award CCF-1845076 and ARO YIP award W911NF1910027.

\bibliography{mybib}
\bibliographystyle{icml2024}

\newpage
\appendix
\onecolumn

\section{Preliminaries on Riemannian Optimization}
	\label{sec:preliminaries_Riemannian}

In this paper, we use the same notations as the common Riemannian optimization literature, see e.g. \cite{absil2008optimization} and \cite{boumal2023introduction}. We refer the readers to \cite{sakai1996riemannian}, \cite{lee2003introduction}, \cite{do1992riemannian}, and \cite{helgason1979differential} for the background knowledge on Riemannian geometry. Below we give some preliminaries on Riemannian optimization.
	
    We call a manifold to be a \textit{Riemannian manifold} if it is equipped with a Riemannian metric $\left(\xi, \eta\right) \mapsto\left\langle\xi, \eta\right\rangle_x \in \mathbb{R}$, where the tangent vectors $\xi$ and $\eta$ are on the tangent space $T_{x}\mathcal{M}$. We denote $T_{x}\mathcal{M}$ as $T_{x}$ when the corresponding manifold is clear from the context. Moreover, we drop the subscript $x$ of the inner product $\langle \cdot , \cdot \rangle_x$ when it is clear from the context. This inner product induces a norm on the tangent space, which is denoted by $\|\cdot\|_{x}$ or $\|\cdot\|$. The geodesic distance generalizes the concept of distance by measuring the shortest path between two points on a curved surface, accounting for the geometry of the manifold. We denote the geodesic distance between $x,y\in\M$ by $d_{\mathcal{M}}(x,y)$ or simply $d(x,y)$. An important concept in Riemannian optimization is the \textit{Riemannian gradient} of a smooth function $f: \mathcal{M} \rightarrow \R$ denoted as $\grad f(x)$ at $x\in \M$. It is defined as the tangent vector satisfying $\langle\operatorname{grad} f(x), \xi_{x}\rangle=\mathrm{D} f(x)\left[\xi_{x}\right], \forall \xi_{x} \in T_{x} \mathcal{M}$. Here $\mathrm{D} f(x)\left[\xi_{x}\right]$ is the directional derivative along the direction $\xi_{x}$. As aforementioned in the introduction, the widely used Riemannian gradient descent algorithms use retractions to map the Riemannian gradient to the manifold. A \textit{retraction} denoted as $\rtr$ is a smooth mapping from the tangent bundle $T\M$ to $\M$ that satisfying the following key properties.
    \begin{description}
		\item[(i)] For each $x\in \M$, define  $\rrtr(x)>0$ to be the `retraction radius' such that within the ball of radius $\rrtr(x)$ around the origin $\mathbf{0}=\mathbf{0}_{x}$, the restriction $\rtr_{x} : T_{x}\mathcal{M} \rightarrow \mathcal{M}$ of $\rtr$ to $T_{x}\mathcal{M}$ is well-defined. 
		
		\item[(ii)] $\rtr_{x}(\mathbf{0})=x$;  The differential of $\rtr_{x}$ at $\mathbf{0}$, $D \rtr_{x} (\mathbf{0})$, is the identity map on $T_{x}\mathcal{M}$. 
	\end{description}
For $x\in \M$ and $\eta\in T_x\M$, the retraction curve $t\mapsto \rtr_{x}(t\eta)$ agrees with the geodesic passing $x$ with initial velocity $\eta$ to the first order. Furthermore, if this retraction curve coincides with the geodesic for all $x\in \M$ and $\eta \in T\M$, then this specific retraction is called an \textit{exponential map}. In fact, the definition of exponential map involves solving a nonlinear ordinary differential equation, which brings computational burden. Hence, using computationally efficient retractions instead of exponential map is more desirable. The widely used retractions including $\rtr_{x}(\eta)=x+\eta$ in Euclidean spaces and $\rtr_{x}(\eta)=\frac{x+\eta}{\lVert x+\eta \rVert}$ on spheres. For more examples and details, we refer the readers to Sec. 4.1 in \cite{absil2008optimization}. A manifold $\M$ is called (geodesically) \textit{complete} if the domain of exponential map is the entire tangent bundle. The exponential map by its definition preserves distance, i.e. $\|\eta\| = d(x,y)$ for $\Exp_x (\eta) = y$. 

A function $g:\M\rightarrow \R$ can be lifted to the tangent spaces via retractions. Namely, consider the composition of $g$ and a retraction $\rtr$, define the \textit{pullback} as $\hat{g} :=g\circ \rtr_{x}:T_{x}\rightarrow \R$. In tBMM, we seek a majorizing surrogate of the lifted marginal loss function to minimize in each iteration. One important observation used in the analysis is the following,
    \begin{align}\label{eq:pullback_derivative}
		\langle \nabla \hat{g}(\mathbf{0}),\, \eta \rangle = D \hat{g}(\mathbf{0})[\eta] = D g(x) [ D \rtr_{x}(\mathbf{0}) [\eta] ] = D g(x) [\eta] = \langle \grad g(x),\, \eta \rangle. 
    \end{align}

At each point $x \in \M$, the Riemannian manifold $\M$ resembles an Euclidean space within a small metric ball with radius $r$ since it is locally diffeomorphic to the tangent space. The \textit{injectivity radius}, denoted as $\rinj(x)$, is defined as the supremum of values of $r$ such that this diffemorphism holds. Namely, $\Exp_x$ is a diffemorphism of $B(x,r)$ and its image on $\M$ when $r\le \rinj(x)$ where $B(x,r)=\{\eta\in T_x \mathcal{M}:\langle \eta,\eta\rangle<r\}\subseteq T_x \mathcal{M}$. Therefore, the inverse exponential map $\Exp^{-1}_x$ is well defined within the small ball of radius $\rinj(x)$. Many widely studied manifolds have uniformly positive injectivity radius, including compact manifolds (see Thm. III.2.3 in \cite{chavel2006riemannian}) and Hadamard manifolds (see Appendix \ref{sec:Had_eg}, \cite{afsari2011riemannian} and Theorem 4.1, p.221 in \cite{sakai1996riemannian}). We call a subset $C\subseteq M$ to be (geodesically) \textit{strongly convex} if there is a unique minimal geodesic $\gamma$ connecting any two points $x,y\in C$ and that $\gamma \subseteq C$. 

For a subset $\Param\subseteq \mathcal{M}$ and $x\in \Param$, define the \textit{tangent cone} $\mathcal{T}_{\Param}(x)$ and the \textit{normal cone} $\mathcal{N}_{\Param}(x)$ at $x$ as 
	\begin{align}\label{eq:def_tangent_normal_cone}
		&	\mathcal{T}_{\Param}\left(x \right):=\left\{u \in T_{x}\M \,\bigg|\,  \textup{$\Exp_{x}\left(t \frac{u}{\lVert u \rVert}\right)\in \Param$ for some $t\in (0, \rinj(x))$} \right\}\cup \{ \mathbf{0} \}, \\
		&	\mathcal{N}_{\Param}\left(x \right):=\left\{u \in T_{x}\M \,\bigg|\, \left\langle u, \eta\right\rangle \leq 0 \,\,  \textup{for all $\eta\in T_{x}\M$ s.t. $\Exp_{x}\left(t \frac{\eta}{\lVert \eta\rVert}\right)\in \Param$ for some $t\in (0, \rinj(x))$} \right\}.
	\end{align}
	Note that	$\mathcal{T}_{\Param}(x)=T_{x}\M$ and $\mathcal{N}_{\Param}(x)=\{ \mathbf{0} \}$ if $x$ is in the interior of $\Param$. When $\Param$ is strongly convex, then the tangent cone $\mathcal{T}_{\Param}(x)$ is a convex cone in the tangent space $T_{x}\M$ (see Prop.1.8 in \cite{cheeger1972structure} and \cite{afsari2011riemannian}).

 The \textit{lifted constraint set} $T_x^* \M$ is defined as 
\begin{align}\label{eq:def_lift_constraints1}
	T^{*}_{x} \mathcal{M} &:= \{ u\in T_{x} \mathcal{M} \,|\,  \textup{$\rtr_x(u)=x'$ for some $x'\in \Param$ with $d(x,x')\le r_0 /2$} \},
\end{align}
 where $r_0$ is the uniform lower bound of injectivity as in \ref{assumption:Optn2}\textbf{(iii)}. See an example of the lifted constraint set in Appendix \ref{sec:lifted_set} and Figure \ref{fig:lifted_set}. In \eqref{eq:def_lift_constraints1}, when exponential map is used as the retraction, the set $T^*_x \M$ can be viewed as the 'lift' of the restricted constraint set $\Param$ within a small ball to the tangent space $T_x\M$, which equals the image $\Exp^{-1}_x (\Param\cap B(x,r_0/2))$ by the inverse exponential map. For the special case when $\M$ is an Euclidean space, then the retraction becomes identity. Therefore if $\Param$ is a convex subset, we have $T^*_x \M = \Param$ is also convex. We remark that if $\Param$ is (geodesically) strongly convex, then the lifted constraint set $T^*_x \M$ is locally well defined, i.e. we can replace $r_0$ in \eqref{eq:def_lift_constraints1} by any value $r' \in (0, r_0)$.

 At different points on a Riemannian manifold, the corresponding tangent spaces are different. The \textit{parallel transport} allows one to transport a tangent vector to another tangent space. For a smooth curve $\gamma:[0,1] \to \M$, the parallel transport along $\gamma$ from the point $x=\gamma(0)$ to the point $y=\gamma(1)$ is denoted as  $\Gamma^{\gamma}_{x\to y}$. We drop the superscript $\gamma$ when it is clear from the context. Intuitively, for a tangent vector $\eta\in T_x\M$, the vector $\Gamma^{\gamma}_{x\to y} \eta$ is a tangent vector in $T_y \M$. The key property of parallel transport is that it is a linear isomorphism that preserves inner product. Namely, 
 \begin{align}
		\left\langle \Gamma^{\gamma}_{\gamma(0)\rightarrow \gamma(t)} \, (\eta) ,\,  \Gamma^{\gamma}_{\gamma(0)\rightarrow \gamma(t)} \, (\zeta) \right\rangle_{\gamma(t)} = \left\langle  \eta,\, \zeta\right\rangle_{\gamma(0)}  \quad \textup{for all $t\in [0,1]$, $\eta,\zeta\in T_{\gamma(0)}$}.
	\end{align}

	\subsection{Notations for Block Riemannian Optimization}
	\label{sec:notations}

In \eqref{eq:def_CROPT_block}, we aim to minimize an objective function $f$ within the product constraint set $\Param=\Theta^{(1)}\times \dots \times \Theta^{(m)}$, where each $\Theta^{(i)}$ is a subset on a Riemannian manifold $\mathcal{M}^{(i)}$. For convenience, we introduce the following notations: For $\param=[\theta^{(1)},\dots,\theta^{(m)}]$, let
\begin{align}
		\grad_{i} f(\param)&:=\textup{Riemmanian gradient of $\theta \mapsto f(\theta^{(1)},\dots,\theta^{(i-1)}, \theta, \theta^{(i+1)},\dots,\theta^{(m)})$},\\ 
		\grad f(\param)&:=[ \grad_{1} f(\param),\dots, \grad_{m} f(\param) ], \\
		 d(\mathbf{x}, \mathbf{y})  &:= \sqrt{\sum_{i=1}^{m} d(x^{(i)}, y^{(i)})^{2}} \quad \textup{for $\mathbf{x}=(x^{(1)},\dots,x^{(m)}), \,  \mathbf{y}=(y^{(1)},\dots,y^{(m)})\in \prod_{i=1}^{m}\M^{(i)}$}.  \label{eq:def_product_dist}
	\end{align}
We remark that one can endow the product manifold $\prod_{i=1}^{m}\M^{(i)}$ a joint Riemannian structure, and therefore the above $\grad f(\param)$ can be viewed as the full Riemannian gradient at $\param$ w.r.t the joint Riemannian structure. However, in the present paper we do not explicitly use the product manifold structure.

Throughout the paper, we denote $(\param_{n})_{n\ge 1}$ as an output of Algorithm \ref{algorithm:BMM} and write $\param_{n}=[\theta_{n}^{(1)},\dots,\theta_{n}^{(m)}]$ for $n\ge 1$. For each $n\ge 1$ and $i=1,\dots,m$, denote the marginal objective function as 
\begin{align}\label{eq:def_f_n_marginal1}
		f_{n}^{(i)}: \theta \mapsto f(\theta_{n}^{(1)},\dots,\theta_{n}^{(i-1)},\theta,\theta_{n-1}^{(i+1)},\dots,\theta_{n-1}^{(m)}),
	\end{align}
 which is at iteration $n$ and block $i$.

 Below we define the \textit{joint smoothness} used in Thm.\ref{thm:RBMM_MtMt}\textbf{(iii)}.
    \begin{definition}[Joint smoothness]\label{def:joint_smooth}
        Let $\M = \M^{(1)}\times \cdots\times \M^{(m)}$ be a product manifold where all the $\M^{(i)}$s are smooth Riemannian manifold. Let $\varphi:\M\to \R$ be a smooth function in each block. We say $\varphi$ is \textit{jointly smooth} if there exists a constant $L'>0$ such that for each distinct $i,j\in \{1,\dots,m\}$ and $(\theta^{(1)},\dots, \theta^{(m)})\in \Param$, whenever $V\in T_{\theta^{(j)}}$ such that $\rtr_{\theta^{(j)}}(V)\in \Theta^{(j)}$, 
				\begin{align}\label{eq:thm_MtMt_smoothness}
					&\left\lVert 
					\begin{matrix} 
						\hspace{-2cm} \grad_{i} \varphi(\theta^{(1)},\dots,   \theta^{(j)},  \dots, \theta^{(m)}) \\
						\hspace{1cm} - \grad_{i} \varphi(\theta^{(1)},\dots, \rtr_{\theta^{(j)}}(V),\dots, \theta^{(m)})  \end{matrix} \right\rVert  \hspace{0.3cm} \le L' \lVert V \rVert.
				\end{align}
    \end{definition}


\section{Relation Between Types of Majorizing Surrogates}
\label{sec:app_majorizer}

 In \cite{li2023convergence}, the authors studied the convergence of \textit{RBMM}, which is a Riemannian block majorization-minimization algorithm using the majorizing surrogates on the manifold. A function $g:\M\rightarrow \R$ is a \textit{majorizing surrogate} of $h$ at $\param$ if 
		\begin{align}
			g(x) \ge h(x) \quad \textup{for all $x\in \M$} \quad \textup{and} \quad g(\param)=h(\param). 
		\end{align}

 We remark that if a function $g:\M\rightarrow \R$ is a majorizing surrogate of $h:\M\rightarrow \R$ at some point $\param\in \M$, then the pullback surrogate $\hat{g}:=g\circ \rtr_{\theta}: T_{\param}\M\rightarrow \R$  is a tangential surrogate of $h$ at $\param$. Indeed, 
	\begin{align}
		&\hat{g}(\eta) = g(\rtr_{\param}(\eta)) \ge h(\rtr_{\param}(\eta))  = \hat{h}(\eta) \,\, ,\,  \forall \eta\in T_{\param}\M, \\
		&\hat{g}(\mathbf{0}) = g(\rtr_{\param}(\mathbf{0})) = g(\param) = h(\param) = h(\rtr_{\param}(\mathbf{0}))  = \hat{h}(\mathbf{0}). 
	\end{align}
Note that minimizing the majorizer $g$ directly on the manifold $\mathcal{M}$ and minimizing its pullback $\hat{g}$ on the tangent space are in general two different optimization problems. 
However, if the manifold $\mathcal{M}$ is Euclidean, then we can take the trivial retraction of translating the base point to the origin, in which case the two viewpoints coincide. Hence in the Euclidean case, tBMM agrees with RBMM in \cite{li2023convergence}. 
Also in the Euclidean setting, tBMM agrees with the standard Euclidean BMM \cite{hong2015unified}. Hence tBMM generalizes the Euclidean BMM. 

\section{Line Search Algorithm}
Below is the optional line search for choosing the step size $\alpha$ in Algorithm \ref{algorithm:BMM},
	\begin{algorithm}[H]
		\small
		\caption{Constrained Riemannian Line Search} 
		\label{alg:R_line_search}
		\begin{algorithmic}[1]
			\STATE \textbf{Input:} $\alpha\in[0,1]$ (initial step size); \, $\theta_{n}^{(i)}\in \mathcal{M}^{(i)}$ (previous block iterate);\, $V_{n}^{(i)}\in \mathcal{T}_{\Theta^{(i)}}(\theta_{n-1}^{(i)})$ (admissible search direction in the tangent cone);\,   $\gamma\in (0,1)$ (shrinkage parameter)
			
			\STATE 
			set\, $\alpha=1$
			\WHILE{  $f_{n}^{(i)}(\rtr_{\theta_{n-1}^{(i)}}(\alpha V_{n}^{(i)}))-f_{n}^{(i)}(\theta_{n-1}^{(i)})>\frac{\rho\alpha}{4}\lVert V_{n}^{(i)}\rVert^{2}$ or $\rtr_{\theta_{n-1}^{(i)}}(\alpha V_{n}^{(i)}) \notin \Theta^{(i)}$}
			\STATE \quad 
			$\alpha=\gamma\alpha$
			\ENDWHILE
			
			\STATE \textbf{output:}  $\alpha$ 
		\end{algorithmic}
	\end{algorithm}

\section{An Example of the Lifted Constraint Set}
\label{sec:lifted_set}

In this section, we provide a concrete example of a constrained problem on the manifold. An illustration of this example is shown in Figure \ref{fig:lifted_set}. Consider the 2-sphere $S^2 = \{x\in \mathbb{R}^{3} : \|x\|=1\}$, which can also be viewed as a simple Stiefel manifold. Consider three points $x,y,z \in S^2$ that $x$ is at the north pole; $y$ and $z$ are on the equator such that $d(x,y)=\frac{\pi}{2}$. Then the region on the sphere bounded by the geodesic triangle $\triangle xyz$ is a geodesic convex constraint set. Let us call this constraint set $\Theta$. Now for a point $\theta \in \Theta \subset S^2$, the corresponding lifted constraint set $T^*_{\theta}$ is one of the following three cases: 1) If $\theta$ is a vertex of the geodesic triangle (e.g. $x$ in Figure \ref{fig:lifted_set}), then the lifted constraint set is a circular sector with central angle $\frac{\pi}{2}$; 2) If $theta$ is at the boundary but not the vertex (e.g. $\theta_1$ in Figure \ref{fig:lifted_set}), then $T^*_{\theta}$ is a half-disk; 3) If $\theta$ is in the interior of the triangle (e.g. $\theta_2$ in Figure \ref{fig:lifted_set}), then $T^*_{\theta}$ is a disk. In all the three cases, the lifted constraint set is convex on the tangent space. In fact, one could generalize this example of the geodesic triangle to the geodesic polyhedron on a submanifold of $\mathbb{R}^n$. One can then show the corresponding lifted constraint sets satisfy our assumptions using the same argument.

\section{Examples of Hadamard Manifolds}
\label{sec:Had_eg}

The complete and simply connected Riemannian manifolds with nonpositive sectional curvature at every point is called \textit{Hadamard manifolds} (\cite{Burago2001ACI} and \cite{burago1992alexandrov}). The injectivity radius is infinity at each point on a Hadamard manifold. In fact, many widely studied spaces are Hadamard manifold and we give some examples below. We refer the readers to \cite{Bacak2014convex} for more details.
\begin{example}[Euclidean spaces]
		\normalfont
		The Euclidean space $\mathbb{R}^n$ with its usual metric is a Hadamard manifold. The sectional curvature at each point is constant and equal to 0.
		\hfill $\blacktriangle$
\end{example}

\begin{example}[Hyperbolic spaces]
		\normalfont
		Equip $\mathbb{R}^{n+1}$ with the $(-1, n)$-inner product given by
		$$
		\langle x, y\rangle_{(-1, n)}:=-x^0 y^0+\sum_{i=1}^n x^i y^i
		$$
		for $x:=\left(x^0, x^1, \ldots, x^n\right)$ and $y:=\left(y^0, y^1, \ldots, y^n\right)$. Denote
		$$
		\mathbb{H}^n:=\left\{x \in \mathbb{R}^{n+1}:\langle x, x\rangle_{(-1, n)}=-1, x_0>0\right\} .
		$$
		Then the inner product $\langle\cdot, \cdot\rangle$ induces a Riemannian metric $g$ on the tangent spaces $T_p \mathbb{H}^n \subset T_p \mathbb{R}^{n+1}$ for each $p \in \mathbb{H}^n$. The sectional curvature of $\left(\mathbb{H}^n, g\right)$ is constant and equal to $-1$ at every point.
		\hfill $\blacktriangle$
	\end{example}

\begin{example}[Manifolds of positive definite matrices]
		\normalfont
		\label{eg:PSD}
		The space $\mathbb{S}_{++}^n$ of all symmetric positive definite matrices of size $n\times n$ with real entries is a Hadamard manifold when equipped with the Riemannian metric given by
		\begin{equation}
			\left\langle\Omega_1, \Omega_2\right\rangle_{\Sigma} \triangleq \frac{1}{2} \operatorname{Tr}\left(\Omega_1 \Sigma^{-1} \Omega_2 \Sigma^{-1}\right) \quad \forall \Omega_1, \Omega_2 \in T_{\Sigma} \mathbb{S}_{++}^n.
		\end{equation}
		\hfill $\blacktriangle$
	\end{example}

\section{Preliminary Lemmas}
   \begin{lemma}\label{lem:bounded_subgradient}
       Let $\psi : \R^n \to \R$ be a convex and Lipschitz continuous function with respect to the norm $\|\cdot\|$ with parameter $L_\psi$. Then for any $x$ in the domain of $\psi$ and any subgradient $\partial \psi$, we have $\|\partial \psi (x)\| \le L_\psi$.
   \end{lemma}
\begin{proof}
    Fix any $x$ in the domain of $\psi$ and any $\partial \psi$. 
    Let 
    \begin{align}
        u^{*}\in \argmax_{u, \|u\|\le 1} \langle u , \partial \psi(x) \rangle. 
    \end{align}
    Let $y = x + u^{*}$. Then $u^{*}=\frac{\partial \psi(x)}{\lVert \partial \psi(x) \rVert}$ and $\langle u^{*},\, \partial \psi(x) \rangle= \lVert  \partial \psi(x) \rVert$, so 
    \begin{align}
        \|\partial \psi(x)\| = \langle u^{*} , \partial \psi (x) \rangle \le \psi(x+u^{*}) - \psi(x) \le L_{\psi} \|u^{*}\| = L_\psi.
    \end{align}
    where the first inequality is by convexity of $\psi$, the second inequality is by Lipschitz continuity of $\psi$ and the last equality is by definition of $y$.
\end{proof}		

\begin{prop}[Euclidean smoothness implies geodesic smoothness on the Stiefel manifold]\label{lem:g_smooth_Stiefel}
    If $f$ is $L$-smooth in Euclidean space $\R^{n\times p}$, then there exists a constant $\tilde{L}_g = M_0^2 L+ M_2 L_N$ such that $f$ satisfying restricted smoothness \eqref{eq:f_Lipschitz_type_grad} with parameter $\tilde{L}_g$ on the Stiefel manifold $\mathcal{V}^{n\times p}$, where $L_N = \max_{x\in \mathcal{V}^{n\times p}} \|\nabla f(x)\|$, $M_0$ is a constant related to the retraction, and $M_2$ is the same constant as in Prop. \ref{prop:cpt_manifold_in_Euclidean}.
\end{prop}
\begin{proof}
    See \cite{chen2021decentralized}, Lemma 2.4 and Appendix C.1.
\end{proof}

\section{Proof for Section \ref{sec:intro}}\label{sec:proof_intro}

    \begin{proof}[\textbf{Proof of Corollary \ref{cor:inexact_RGD}}]
 Under the assumptions of Cor. \ref{cor:inexact_RGD}, the inexact retraction step in \eqref{eq:RGD_inexact} can be viewed as an exact retraction on $V_n$. Therefore, the inexactness of the second step in \eqref{eq:RGD_inexact} is transferred to the first step. Namely, let 
\begin{align}
    \hat{g}_n(\eta)= \varphi(\param_{n-1})+\langle \grad \varphi(\param_{n-1}),\eta\rangle + \frac{L}{2}\|\eta\|^2.
\end{align}
Then inexact RGD can be viewed as tMM with tangential surrogate $\hat{g}_n(\eta)$. The exact solution of minimizing $\hat{g}_n$ is $-\frac{1}{2L}\grad \varphi(\param_{n-1})$ and the inexact solution used for inexact RGD is $V_n$. Recall $\Delta_n = \hat{g}_n(\alpha_n V_n) - \hat{g}_n (-\frac{1}{2L}\grad \varphi(\param_{n-1}))$ is summable. Hence \ref{assumption:A0_optimal_gap}\textbf{(ii)} is satisfied. Note \ref{assumption:A0_optimal_gap}\textbf{(i)} is satisfied by the assumptions in Cor. \ref{cor:inexact_RGD} and \ref{assumption:A0_optimal_gap}\textbf{(iii)} is not needed in the analysis when $\psi=0$ (see the proof of Theorem \ref{thm:RBMM_MtMt} in Section \ref{sec:pf_lemma}). By definition of $\hat{g}$, \ref{assumption:Optn2}\textbf{(i),(ii)} are satisfied. \ref{assumption:Optn2}\textbf{(iii)} again is satisfied by assumptions in Cor. \ref{cor:inexact_RGD}. Hence Thm. \ref{thm:RBMM_MtMt} holds. The convergence and complexity results follow.

    \end{proof}

\section{Proof for Section \ref{sec:examples}}\label{sec:proof_eg}

\begin{proof}[\textbf{Proof of Corollary \ref{cor:nonsmooth_Stiefel}}]
In order the complexity results in Thoerem \ref{thm:RBMM_MtMt} hold, we need to show assumptions \ref{assumption:A0_optimal_gap} and \ref{assumption:Optn2} hold. \ref{assumption:A0_optimal_gap}\textbf{(i)} holds by the assumptions on objective function of problem \eqref{eq:nonsmooth_st}. \ref{assumption:A0_optimal_gap}\textbf{(ii)} holds trivially. \ref{assumption:A0_optimal_gap}\textbf{(iii)} holds by Proposition \ref{prop:cpt_manifold_in_Euclidean}. To see \ref{assumption:Optn2} holds, let
\begin{align}
\hat{g}(V):= \phi(\param_{n-1}) + \langle \grad \phi(\param_{n-1}) , V\rangle + \lambda \|V\|^2_F .  
\end{align}
Then since $\hat{g}$ is a quadratic function, \ref{assumption:Optn2}\textbf{(i)}, \textbf{(ii)} hold. \ref{assumption:Optn2}\textbf{(iii)} holds since Stiefel manifold is complete and \eqref{eq:nonsmooth_st} is an unconstrained problem. Namely, the constrained set is the manifolds itself, i.e. $\Theta = \mathcal{V}^{n\times p}$. Lastly, $\hat{g}(V)$ is indeed a tangential majorizer by the same analysis in Section \ref{sec:brpl} and Prop. \ref{lem:g_smooth_Stiefel}. Therefore, all the assumptions of Theorem \ref{thm:RBMM_MtMt} are satisfied.
\end{proof}

\section{Convergence Analysis for tBMM}
\label{sec:pf_lemma}

In this section, we prove Theorem \ref{thm:RBMM_MtMt}. We will first prove the statement for the single-block case ($m=1$) for notational simplicity. The main ingredient in our analysis is the following descent lemma for constrained tBMM.

		\begin{lemma}[An inexact descent lemma for constrained tBMM]
			\label{lem:CRMM_descent}
			Suppose $\Param$ is a strongly convex subset of a Riemannian manifold $\mathcal{M}$. Fix a function $f:\mathcal{M}\rightarrow \R$ and $\param\in \Param$. Denote the pullback objective $\hat{f}:=f\circ \rtr_{\param}:T_{\param}\M\rightarrow \R$. Suppose we have a function $\hat{g}: T_{\param}\M \rightarrow \R$ such that the following hold:
			\begin{description}[itemsep=0.1cm]
				\item[(i)] (Tangential marjorization) $\hat{g}$ is a majorizing surrogate of the pullback smooth part of the objective $\hat{\varphi}$:
				\begin{align}
					\hat{g}(\eta) \ge \hat{\varphi}(\eta) \,\, \textup{for all $\eta\in T_{\param}\M$ and $\hat{g}(\mathbf{0})=\hat{\varphi}(\mathbf{0})$}. 
				\end{align}
				
				\item[(ii)] (Strong convexity of surrogate) The surrogate $\hat{g}:T_{\param} \rightarrow \R$ is $\rho$-strongly convex for some $\rho>0$: for all $\eta\in T_{\param}\M$,
				\begin{align}
					\hat{g}(\eta)  \ge \hat{g}(\mathbf{0}) + \langle \nabla \hat{g}(\mathbf{0}),\, \eta \rangle - \frac{\rho}{2}\lVert \eta \rVert^{2}. 
				\end{align}
			\end{description}
			Let $T_{\param}^*$ denote the lifted constrained set $\Param$ at $\param$ (see \eqref{eq:def_lift_constraints}). Define $V^{\star}\in \argmin_{\eta\in T_{\param}^*} \left(\hat{g}(\eta)+ \psi(\theta+\eta)\right)$.  Fix $V\in T_{\param}^*$ denote $\Delta:= \hat{g}(V) +\psi(\theta+V)- \hat{g}(V^{\star})-\psi(\theta+V^{\star})$. 
			Then for each $\alpha\in [0, 1]$, $\rtr_{\param}(\alpha V)\in \Param$ and 
			\begin{align}\label{eq:descent_lemma}
				&f(\rtr_{\param}(\alpha V)) - f(\param)\\ \le& -\frac{\rho \alpha}{2}\lVert V^{\star} \rVert^{2}-\frac{1}{2}\rho \alpha\left(1-\frac{2L_\psi M_2 +\rho}{\rho}\alpha\right)\lVert V \rVert^{2}+\alpha \Delta.
			\end{align}
		\end{lemma}



		\begin{proof}[\textbf{Proof of Lemma \ref{lem:CRMM_descent}}]
			That $\rtr_{\param}(\alpha V)$ is well-defined and belongs to $\Param$ is due to the definitions of the retraction and the lifted constrained set. Hence we only need to show \eqref{eq:descent_lemma}. 
			
			Let $G:= \hat{g}+\psi :T_{\theta} \rightarrow \R$. Then since $\hat{g}$ is strongly convex and $\psi$ is convex, we have $G$ is strongly convex. Recall $T^*_{\param}$ is a convex subset of tangent space $T_{\param}$ by \ref{assumption:Optn2}\textbf{(iii)}. The optimization problem that defines $V^{\star}$ is to minimize a strongly convex function in a convex subset of Euclidean space, so it admits a unique solution $V^{\star}$. Then by the first order optimality condition and since  $\mathbf{0}\in T_{\param}$, 
			\begin{align}
				\langle \nabla  \hat{g}(V^{\star} ) + \partial \psi(\param + V^{\star}),\, \mathbf{0} -V^{\star} \rangle \ge 0. 
			\end{align}
			Since $G: T_{\theta} \rightarrow \R$  is  $\rho$-strongly convex, 
			\begin{align}
				& G(\mathbf{0})-G(V^{\star})\\
 \geq &\langle \nabla  \hat{g}(V^{\star} ) + \partial \psi(\param + V^{\star}), \mathbf{0} -V^{\star} \rangle + \frac{\rho}{2}\|V^{\star}\|^{2} \\
 \ge  &\frac{\rho}{2}\|V^{\star}\|^{2}.
			\end{align}
			Hence 
			\begin{align}
				&G(\mathbf{0}) - G(V) \\= &\left( G(\mathbf{0}) - G(V^{\star})  \right) + \left( G(V^{\star}) - G(V)  \right)  \ge  \frac{\rho}{2}\|V^{\star}\|^{2}-\Delta.
			\end{align}
			Again using the $\rho$-strong convexity of $G: T_{\param} \rightarrow \R$ and the above inequality, for $\alpha\in[0,1]$,
			\begin{align}
				G(\alpha V)  - G(\mathbf{0}) 
				& = G(\alpha V + (1-\alpha) \mathbf{0} )  - G(  \mathbf{0}) \\
				& \leq \alpha G(V)+ (1-\alpha)G(\mathbf{0}) + \frac{\rho \alpha(1-\alpha)}{2}\|V\|^2 - G(\mathbf{0}) \\
				&=\alpha(G(V)-G(\mathbf{0}))+\frac{\rho \alpha(\alpha-1)}{2}\|V\|^2 \\
				&\leq -\frac{\rho \alpha}{2}\lVert V^{\star} \rVert^{2}-\frac{\rho \alpha(1-\alpha)}{2}\lVert V \rVert^{2}+\alpha \Delta.
				\label{eq:g_descent_before_retraction} 
			\end{align}
			In order to conclude,  note the following: 
			\begin{align}
				f(\rtr_{\param}(\alpha V)) - f(\param)  & = \hat{\varphi}(\alpha V) + \psi(\rtr_{\param}(\alpha V)) - \hat{\varphi}(\mathbf{0}) - \psi(\param)\\
                    &\overset{(a)}{\le} \hat{\varphi}(\alpha V) + \psi(\alpha V) + L_\psi \|\rtr_{\param}(\alpha V)- \alpha V\| - \hat{\varphi}(\mathbf{0}) - \psi(\param)\\
				&  \overset{(b)}{\le} G(\alpha V) - G(\mathbf{0}) + L_\psi M_2 \|\alpha V\|^2 \\
				&  \overset{(c)}{\le} -\frac{\rho \alpha}{2}\lVert V^{\star} \rVert^{2}-\frac{1}{2}\rho \alpha\left(1-\frac{2L_\psi M_2 +\rho}{\rho}\alpha\right)\lVert V \rVert^{2}+\alpha \Delta.
				\label{eq:g_descent_before_retraction2} 
			\end{align}
			Namely, (a) follows from the Lipschitz continuity of $\psi$, (b) uses Proposition \ref{prop:cpt_manifold_in_Euclidean}, (c) uses \eqref{eq:g_descent_before_retraction}. This shows \eqref{eq:descent_lemma}, as desired. 
		\end{proof}

  A direct consequence of Lemma \ref{lem:CRMM_descent} is the boundedness of iterates, which is stated in the following proposition.
\begin{prop}[Boundedness of iterates]
		\label{prop:boundedness_iterates_opt2}
		Under \ref{assumption:A0_optimal_gap} and  \ref{assumption:Optn2}, the set $\{\param_n : n\geq1\}$ is bounded.
	\end{prop}
\begin{proof}[\textbf{Proof of Proposition \ref{prop:boundedness_iterates_opt2}}]
		Let $m\sum_{n=1}^{\infty}\Delta_n =T<\infty$. By Lemma \ref{lem:CRMM_descent} we immediately have, for all $n\ge 1$, $f(\param_n)\le f(\param_{n-1})+m\Delta_{n-1}$ and therefore $f(\param_n)\le f(\param_{0})+m\sum_{i=1}^{n-1}\Delta_{i}\le f(\param_{0})+T$ . Consider $K=\{\param \in \Param: f(\param)\leq f(\param_0)+T\}$, by \ref{assumption:A0_optimal_gap}\textbf{(i)}, $K$ is compact. Hence the set $\{\param_n : n\geq1\}$ is bounded.
	\end{proof}

	The following lemma states the local property of retractions, which is essential in the asymptotic convergence analysis.
	
	\begin{lemma}\label{lem:retaction_d_upper_bd}
		For $\param \in \mathcal{M}$, $V\in T_{\param}$ and $\alpha>0$, when $\alpha \|V\|$ is small, we have
		\begin{align}
			d\left(\rtr_{\param}(\alpha V),\param\right) = O(\alpha \lVert V \rVert).
		\end{align}
	\end{lemma}
\begin{proof}[\textbf{Proof of Lemma \ref{lem:retaction_d_upper_bd}}]
		By triangle inequality we have
		\begin{align}
			d\left(\rtr_{\param}(\alpha V),\param\right) &\le d\left(\rtr_{\param}(\alpha V),\Exp_{\param}(\alpha V)\right) + d\left(\Exp_{\param}(\alpha V),\param\right) \\ 
			&= o(\alpha \lVert V \rVert ) + \|\alpha V\| \\
			&=O(\alpha \lVert V \rVert),
		\end{align}
		where the first equality is by definition of the exponential map and properties of retraction, see \cite{absil2012projection}.
	\end{proof}
	

Now we are ready to prove Theorem \ref{thm:RBMM_MtMt} and we first prove it for $m=1$. Then in the following proof, we will omit all superscripts $(i)$ in Algorithm \ref{algorithm:BMM} for indicating the blocks since we are in the single-block case. 
	\begin{proof}[\textbf{Proof of Theorem \ref{thm:RBMM_MtMt} for $m=1$}]
		
			By the hypothesis, the step-size $\alpha_{n}$ for $n\ge 1$ is set as $\alpha_n = \alpha \le 1$.
			Therefore, for each $n\ge 1$, $\rtr_{\param_{n-1}}(\alpha V_{n})$ is well-defined and belongs to $\Param$.

			Let $V_{n}^{\star}=\argmin_{\eta\in T^*_{\param_{n-1}}} \left(G_n(\eta):=\hat{g}_{n}(\eta)+ \psi(\param_{n-1} +\eta)\right)$. By Lemma \ref{lem:CRMM_descent} and the hypothesis that $\alpha \le \frac{\rho}{\rho+2L_\psi M_2}$, for each $N\ge 0$, 
			\begin{align}
				0\le \frac{\alpha\rho}{2}  \sum_{n=0}^{N}   \lVert V_{n}^{\star} \rVert^{2}  &\le \sum_{n=0}^{N} f(\param_{n}) - f(\param_{n+1}) + \Delta_{n} \\
				&\le  f(\param_{0}) - f^{*} + \sum_{n=1}^{\infty} \Delta_{n} <\infty.
			\end{align}
			It follows that $\lVert V_{n}^{\star} \rVert= o(1)$ and 
			\begin{align}\label{eq:CRMM_iter_complexity_bd1}
				\min_{0\le k \le N} \,   \lVert V_{k}^{\star} \rVert^{2} \le  \frac{2 (f(\param_{0}) - f^{*} + \sum_{n=1}^{\infty} \Delta_{n})}{\rho \alpha (N+1)}. 
			\end{align}

			Recall that $V_{n}^{\star}$ is the minimizer of the tangential surrogate $\hat{g}_{n}$ on the lifted constraint set $T^*_{\param_{n-1}}\subseteq T_{\param_{n-1}}\M$, which is convex by \ref{assumption:Optn2}\textbf{(iii)}. Hence 
			by the first-order optimality condition 
\begin{align}
\langle \nabla \hat{g}(V_n^\star) + \Proj_{T_{\theta_{n-1}}} \partial \psi(\theta_{n-1} + V_n^\star) \;,\; \eta - V_n^\star\rangle \ge 0 \quad \textup{for all $\eta \in T^{*}_{\theta_{n-1}}$}.
\end{align}
Therefore for all $\eta \in T^{*}_{\theta_{n-1}}$ with $\|\eta\|\le 1$ we have 
\begin{align}\label{eq:opt_bound1}
\langle \nabla \hat{g}_{n}(V_{n}^{\star}) + \Proj_{T_{\theta_{n-1}}} \partial \psi(\theta_{n-1} + V_n^\star) ,\, \eta\rangle  &\ge \langle \nabla \hat{g}_{n}(V_{n}^{\star}) + \Proj_{T_{\theta_{n-1}}} \partial \psi(\theta_{n-1} + V_n^\star) ,\,  V_{n}^{\star} \rangle  \\
&\ge -\|\nabla \hat{g}_n (V_n^\star) + \Proj_{T_{\theta_{n-1}}} \partial \psi(\theta_{n-1} + V_n^\star) \| \cdot \|V_n^\star\| 
\end{align}
Next, we bound the term $\|\nabla \hat{g}_n (V_n^\star) + \Proj_{T_{\theta_{n-1}}} \partial \psi(\theta_{n-1} + V_n^\star) \|$ by a linear term with respect to $\|V_n^\star\|$. Note by Lemma \ref{prop:boundedness_iterates_opt2}, for all $n\ge 1$ there exists constant $C$ such that $\|\grad \varphi(\param_{n-1})\|\le C$. Since $\hat{g}_{n}$ majorizes $\hat{\varphi}$ on $T_{\param}\M$ and and is exact at $\mathbf{0}$, and also noting \eqref{eq:pullback_derivative}, we get $\nabla \hat{g}_{n}(\mathbf{0}) = \nabla \hat{\varphi}(\mathbf{0}) = \grad \varphi(\param_{n-1})$. Moreover, denoting
			\begin{align}
				\zeta_{n}:= \nabla \hat{g}_{n}(\mathbf{0}) - \nabla \hat{g}_{n}(V_{n}^{\star}) ,
			\end{align}
			by the restricted $L_{n}$-smoothness of $\hat{g}_{n}:T_{\param_{n-1}}\rightarrow \R$, 
			\begin{align}
				\lVert \zeta_{n} \rVert \le L_{n} \lVert V_{n}^{\star} \rVert \le L \lVert V_{n}^{\star} \rVert.
			\end{align}
   The above and triangle inequality gives
   \begin{align}\label{eq:opt_bound2}
       \|\nabla \hat{g}_n (V_n^\star)\| \le \|\grad \varphi(\param_{n-1})\|+ L \lVert V_{n}^{\star} \rVert \le C+L \lVert V_{n}^{\star} \rVert.
   \end{align}
Then, note
\begin{align}\label{eq:opt_bound3}
    \|\Proj_{T_{\theta_{n-1}}} \partial \psi(\theta_{n-1} + V_n^*) \| \le \|\partial \psi(\theta_{n-1} + V_n^*) \|\le L_\psi
\end{align}
where the last inequality follows from Lemma \ref{lem:bounded_subgradient}. Hence, \eqref{eq:opt_bound1} together with \eqref{eq:opt_bound2} and \eqref{eq:opt_bound3} gives
\begin{align}
    \langle \nabla \hat{g}_{n}(V_{n}^{\star}) + \Proj_{T_{\theta_{n-1}}} \partial \psi(\theta_{n-1} + V_n^\star) ,\, \eta\rangle  &\ge -(C'+L\|V_n^\star\|)\cdot \|V_n^\star\|
\end{align}
where $C'= C+L_\psi$.
			It follows
			\begin{align}
			&\langle \grad \varphi(\param_{n-1}) +\Proj_{T_{\theta_{n-1}}} \partial \psi(\theta_{n-1} + V_n^\star) ,\eta\rangle \\
   &= \langle \grad \varphi(\param_{n-1})- \nabla \hat{g}_n (V_n^\star),\eta\rangle + \langle \nabla \hat{g}_n (V_n^\star)+ \Proj_{T_{\theta_{n-1}}} \partial \psi(\theta_{n-1} + V_n^\star),\eta\rangle \\
			&\ge -\|\zeta_n\|\cdot \|\eta\| -C'\|V_n^\star\| - L\|V_n^\star\|^{2}\\
			&\ge -(L +C')\|V_n^\star\|-L\|V_n^\star\|^{2}
			\end{align}
			for all $\eta \in T^*_{\param_{n-1}}$ such that $\lVert \eta \rVert\le 1$.
Therefore, \eqref{eq:CRMM_iter_complexity_bd1} and the above give
			\begin{align}
				&\min_{1\le k \le N} \left[ -\inf_{\eta\in T^*_{\param_{k-1}} ,\, \lVert\eta \rVert\le 1}\langle \grad \varphi(\param_{k-1}) +\Proj_{T_{\theta_{k-1}}} \partial \psi(\theta_{k-1} + V_k^\star),\frac{\eta}{\min\{ 1,r_{0}\}}\rangle \right] \\ 
    &\hspace{3cm}\le \frac{L +C'}{\min\{ 1,r_{0}\}} \sqrt{ \frac{2 (f(\param_{0}) - f^{*} + \sum_{n=1}^{\infty} \Delta_{n})}{\rho \alpha (N+1)} }
     +  \frac{L}{\min\{ 1,r_{0}\}} \frac{2 (f(\param_{0}) - f^{*} + \sum_{n=1}^{\infty} \Delta_{n})}{\rho \alpha (N+1)}.
			\end{align}
   When $N\ge \frac{L^2}{(L+C')^2}\frac{2 (f(\param_{0}) - f^{*} + \sum_{n=1}^{\infty} \Delta_{n})}{\rho \alpha}$, the above gives
\begin{align}
				&\min_{1\le k \le N} \left[ -\inf_{\eta\in T^*_{\param_{k-1}} ,\, \lVert\eta \rVert\le 1}\langle \grad \varphi(\param_{k-1}) +\Proj_{T_{\theta_{k-1}}} \partial \psi(\theta_{k-1} + V_k^\star),\frac{\eta}{\min\{ 1,r_{0}\}}\rangle \right] \\
    &\hspace{9cm}\le \frac{2(L +C')}{\min\{ 1,r_{0}\}} \sqrt{ \frac{2 (f(\param_{0}) - f^{*} + \sum_{n=1}^{\infty} \Delta_{n})}{\rho \alpha (N+1)} }.
			\end{align}
   
			This shows \textbf{(i)}. Note that \textbf{(iii)} is equivalent to \textbf{(i)} for the single-block ($m=1$) case.

		Next, we show \textbf{(ii)}. Recall $\psi=0$, so $f=\varphi$ and is continuously differentiable. For asymptotic stationarity, suppose a subsequence of the iterates $(\param_{n_{k}})_{k\ge 1}$ converges to some limit point $\param_{\infty}\in \Param$. Note by first-order optimality of $V_n^\star$ and the $\rho$-strongly convexity of $\hat{g}_n$,
			\begin{align}
			\Delta_n=\hat{g}_n(V_n)-\hat{g}_n(V_n^{\star}) \geq \langle \nabla  \hat{g}_n(V_n^{\star} ), V_n -V_n^{\star} \rangle + \frac{\rho}{2}\|V^{\star}\|^{2} \ge  \frac{\rho}{2}\|V_n-V_n^{\star}\|^{2}.
			\end{align}
   Since $\sum_{n=1}^{\infty}\Delta_n< \infty$ by \ref{assumption:A0_optimal_gap}, we have $\|V_n-V_n^{\star}\|=o(1)$. Hence by triangle inequality $\lVert V_{n} \rVert \le \|V_n-V_n^{\star}\| + \|V_n^{\star}\|=o(1)$.
  Therefore, we have $\lVert \param_{n} - \param_{n-1} \rVert = o(1)$ by Lemma \ref{lem:retaction_d_upper_bd}. Hence $\param_{n_{k}-1}\rightarrow \param_{\infty}$ as $k\rightarrow\infty$. For each $\param\in \Param$, let $B_{\param}$ denote the metric ball of radius half of the injectivity radius $\rinj(\param)$ centered at $\param$. Fix $\param'\in \Param\cap B_{\param_{\infty}}$ be arbitrary. Note that $\param'\in \Param\cap B_{\param_{\infty}} \cap B_{\param_{n_{k}}-1} $ for all  sufficiently large $k$ since $\param_{n}\in \Param$ for all $n\ge 1$ and $\param_{n_{k}-1}\rightarrow \param_{\infty}$. Denote by $\Gamma_{k}$ the parallel transport $\Gamma_{\param_{n_{k}-1}\rightarrow \param_{\infty}}$. Then 
			\begin{align}
				\langle - \Gamma_{k}\left(  \grad f(\param_{n_{k}-1})  \right) ,\,   \Gamma_{k} \big( \eta_{\param_{n_{k}}-1}(\param') \big) \rangle   =   \langle -\grad f(\param_{n_{k}-1}) ,\, \eta_{\param_{n_{k}}-1}(\param') \rangle \le  0.
			\end{align}
			Now since $f$ is continuously differentiable, we have $\grad f(\param_{n_{k}-1}) \rightarrow \grad f(\param_{\infty})$ as $k\rightarrow\infty$. Also, by the continuity of Riemannian metric, the left-hand side above converges to $\langle -\grad f(\param_{\infty}) ,\, \eta_{\param_{\infty}}(\param') \rangle$, so we obtain 
			\begin{align}
				\langle -\grad f(\param_{\infty}) ,\, \eta_{\param_{\infty}}(\param') \rangle \le  0.
			\end{align}
			Since $\param'\in \Param\cap B_{\param_{\infty}}$ is arbitrary, the above show that $\param_{\infty}$ is a stationary point of $f$ over $\Param$.
	\end{proof}

Now we prove Theorem \ref{thm:RBMM_MtMt} for the multi-block case $m\ge 2$. The proof is almost identical to the single-block case. 
	
\begin{proof}[\textbf{Proof of Theorem \ref{thm:RBMM_MtMt} for $m\ge 2$}]
		By the hypothesis, the step-size $\alpha_{n}^{(i)}$ for $n\ge 1$ and $i\in \{1,\dots,m\}$ satisfies $ \alpha_n^{(i)}= \alpha\le 1$.
		Therefore, for each $n\ge 1$, $\rtr_{\theta_{n-1}^{(i)}}(\alpha V^{(i)}_{n})$ is well-defined and belongs to $\Theta^{(i)}$. 
		
		Let  $V_{n}^{( i \star) }=\argmin_{\eta\in T^*_{\theta_{n-1}^{(i)}}} \left(G^{(i)}_n(\eta):=\hat{g}^{(i)}_{n}(\eta)+ \psi^{(i)}_n(\theta^{(i)}_{n-1} +\eta)\right)$. By Lemma \ref{lem:CRMM_descent} and the hypothesis that $\alpha \le \frac{\rho}{\rho+2L_\psi M_2}$, for each $N\ge 0$, 
		\begin{align}
			0< \frac{\alpha \rho}{2}  \sum_{n=0}^{N} \sum_{i=1}^{m}    \lVert V_{n}^{(i\star)} \rVert^{2} 
			&\le \sum_{n=0}^{N} \sum_{i=1}^{m}f_n^{(i)}(\theta_{n-1}^{(i)})-f_n^{(i)}(\theta_{n}^{(i)}) + \Delta_{n} \\
			&\le \sum_{n=0}^{N} f(\param_{n}) - f(\param_{n+1}) + m \Delta_{n} \\
			&\le  f(\param_{0}) - f^{*} + m\sum_{n=1}^{\infty} \Delta_{n}=:M <\infty.
		\end{align}
		It follows that $\sum_{i=1}^{m} \lVert V_{n}^{(i\star)} \rVert= o(1)$ and 
		\begin{align}\label{eq:CRMM_iter_complexity_bd1_gen}
			\min_{0\le k \le N} \, \sum_{i=1}^{m}   \lVert V_{k}^{(i\star)} \rVert^{2}  \le  \frac{2M}{\rho \alpha  N}. 
		\end{align}

		Recall that $V_{n}^{(i\star)}$ is the minimizer of the surrogate $G^{(i)}_{n}$ on the lifted constrained set $T^*_{\theta_{n-1}^{(i)}}\subseteq T_{\theta_{n-1}^{(i)}}\M^{(i)}$, which is convex by \ref{assumption:Optn2}\textbf{(iii)}. Hence 
			by the first-order optimality condition 
\begin{align}
\langle \nabla \hat{g}^{(i)}_n(V_n^\star) + \Proj_{T_{\theta^{(i)}_{n-1}}} \partial \psi^{(i)}_n(\theta^{(i)}_{n-1} + V_n^{(i\star)}) \;,\; \eta - V_n^{(i\star)}\rangle \ge 0 \quad \textup{for all $\eta \in T^{*}_{\theta^{(i)}_{n-1}}$}.
\end{align}
Therefore for all $\eta \in T^{*}_{\theta^{(i)}_{n-1}}$ with $\|\eta\|\le 1$ we have 
\begin{align}
\langle \nabla \hat{g}^{(i)}_{n}(V_{n}^{(i\star)}) + \Proj_{T_{\theta^{(i)}_{n-1}}} \partial \psi^{(i)}_n(\theta^{(i)}_{n-1} + V_n^{(i\star)}) ,\, \eta\rangle  &\ge \langle \nabla \hat{g}^{(i)}_{n}(V_{n}^{(i\star)}) + \Proj_{T_{\theta^{(i)}_{n-1}}} \partial \psi^{(i)}_n(\theta^{(i)}_{n-1} + V_n^{(i\star)}) ,\,  V_{n}^{(i\star)} \rangle  \\
&\ge -\|\nabla \hat{g}^{(i)}_n (V_n^{(i\star)}) + \Proj_{T_{\theta^{(i)}_{n-1}}} \partial \psi^{(i)}_n(\theta^{(i)}_{n-1} + V_n^{(i\star)}) \| \cdot \|V_n^{(i\star)}\|\\
&\ge -(C'+L\|V_n^{(i\star)}\|)\cdot \|V_n^{(i\star)}\|\label{eq:CRMM_fopt_bd1}
\end{align}
for some constant $C'>0$, which can be determined following the same line of $m=1$ case.

Since $\hat{g}_{n}^{(i)}$ majorizes $\hat{\varphi}_{n}^{(i)}$ on $T_{\theta_{n-1}}\M^{(i)}$ and is exact at $\theta_{n-1}^{(i)}$, and also noting  and also noting \eqref{eq:pullback_derivative}, we get $\nabla \hat{g}_{n}^{(i)}(\mathbf{0}) = \nabla \hat{\varphi}_{n}^{(i)}(\mathbf{0}) = \grad \varphi_{n}^{(i)}(\theta_{n-1}^{(i)})$.    Moreover, denoting
		\begin{align}\label{eq:def_zeta_n_i}
			\zeta_{n}^{(i)}:= \nabla \hat{g}_{n}^{(i)}(\mathbf{0}) - \nabla \hat{g}_{n}^{(i)}(V_{n}^{(i\star)}),
		\end{align}
		by the restricted $L_{n}$-smoothness of $\hat{g}_{n}:T_{\param_{n-1}}\rightarrow \R$, 
		\begin{align}\label{eq:CRMM_iter_complexity_bd22}
			\lVert \zeta_{n}^{(i)} \rVert \le L_{n}^{(i)} \lVert V_{n}^{(i\star)} \rVert  \le L\lVert V_{n}^{(i\star)} \rVert.
		\end{align}
		Therefore it follows
			\begin{align}
			&\langle \grad \varphi_n^{(i)}(\theta^{(i)}_{n-1})+\Proj_{T_{\theta^{(i)}_{n-1}}} \partial \psi_n^{(i)}(\theta^{(i)}_{n-1}+V_n^{(i)}),\eta\rangle\\
   &= \langle \grad \varphi_n^{(i)}(\theta^{(i)}_{n-1})- \nabla \hat{g}_n (V_n^{(i\star)}),\eta\rangle + \langle \nabla \hat{g}_n (V_n^{(i\star)})+\Proj_{T_{\theta^{(i)}_{n-1}}} \partial \psi_n^{(i)}(\theta^{(i)}_{n-1}+V_n^{(i)}),\eta\rangle \\
			&\ge -\|\zeta^{(i)}_n\|\cdot \|\eta\| -(C'+L\|V_n^{(i\star)}\|)\|V_n^{(i\star)}\| \\
			&\ge -(L +C')\|V_n^{(i\star)}\|-L\|V_n^{(i\star)}\|^2
			\end{align}
			for all $\eta \in T^*_{\theta^{(i)}_{n-1}}$ and $\|\eta\|\le 1$.
   Thus it follows that
		\begin{align}
		\label{eq:CRMM_iter_complexity_bd3}
			&\sum_{i=1}^{m}  - \inf_{\eta \in T^*_{\theta^{(i)}_{n-1}},\|\eta\|\le 1}\langle \grad \varphi_n^{(i)}(\theta^{(i)}_{n-1})+\Proj_{T_{\theta^{(i)}_{n-1}}} \partial \psi_n^{(i)}(\theta^{(i)}_{n-1}+V_n^{(i)}),\frac{\eta}{\min\{r_0,1\}}\rangle  \\ &\hspace{3cm}\le \frac{L +C'}{\min\{r_0,1\}} \sum_{i=1}^{m} \lVert V_{n}^{(i\star)} \rVert +\frac{L}{\min\{r_0,1\}}\sum_{i=1}^{m} \lVert V_{n}^{(i\star)} \rVert^2.
		\end{align}

		Therefore, the above with \eqref{eq:CRMM_iter_complexity_bd1_gen} implies 
		\begin{align}\label{eq:CRMM_iter_complexity_bd4}
			&\min_{1\le n \le N} \sum_{i=1}^{m}   - \inf_{\eta \in T^*_{\theta^{(i)}_{n-1}},\|\eta\|\le 1}\langle \grad \varphi_n^{(i)}(\theta^{(i)}_{n-1})+\Proj_{T_{\theta^{(i)}_{n-1}}} \partial \psi_n^{(i)}(\theta^{(i)}_{n-1}+V_n^{(i)}),\frac{\eta}{\min\{r_0,1\}}\rangle \\  &\hspace{8cm}\le \frac{L +C'}{\min\{r_0,1\}}\sqrt{\frac{2M}{ \rho \alpha N}}
            + \frac{L}{\min\{r_0,1\}}\frac{2M}{ \rho \alpha N}.
		\end{align}
  When $N\ge \frac{L^2}{(L+C')^2}\frac{2M}{\rho\alpha}$, the above gives 
\begin{align}
			\min_{1\le n \le N} \sum_{i=1}^{m}   - \inf_{\eta \in T^*_{\theta^{(i)}_{n-1}},\|\eta\|\le 1}\langle \grad \varphi_n^{(i)}(\theta^{(i)}_{n-1})+\Proj_{T_{\theta^{(i)}_{n-1}}} \partial \psi_n^{(i)}(\theta^{(i)}_{n-1}+V_n^{(i)}),\frac{\eta}{\min\{r_0,1\}}\rangle  \le& \frac{2(L +C')}{\min\{r_0,1\}}\sqrt{\frac{2M}{ \rho \alpha N}}.
		\end{align}
  
		This shows \textbf{(i)}.


		Next, we show \textbf{(ii)}. Recall $\psi=0$, so $f=\varphi$ and is continuously differentiable. For asymptotic stationarity, suppose a subsequence of the iterates $(\param_{n_{k}})_{k\ge 1}$ converges to some limit point $\param_{\infty}\in \Param$. Note by first-order optimality of $V_n^{(i\star)}$ and the $\rho$-strongly convexity of $\hat{g}_n$,
			\begin{align}
                \label{eq:tangent_optimal_gap}
			\Delta_n=\hat{g}^{(i)}_n(V_n)-\hat{g}^{(i)}_n(V_n^{(i\star)}) \geq \langle \nabla  \hat{g}^{(i)}_n(V_n^{(i\star)} ), V^{(i)}_n -V_n^{(i\star)} \rangle + \frac{\rho}{2}\|V^{(i\star)}\|^{2} \ge  \frac{\rho}{2}\|V^{(i)}_n-V_n^{(i\star)}\|^{2}.
			\end{align}
   Since $m\sum_{n=1}^{\infty}\Delta_n\le \infty$ by \ref{assumption:A0_optimal_gap}, we have $\sum_{i=1}^{m}\|V^{(i)}_n-V_n^{(i\star)}\|=o(1)$. Hence by triangle inequality $\sum_{i=1}^{m}\lVert V^{(i)}_{n} \rVert \le \sum_{i=1}^{m}\|V^{(i)}_n-V_n^{(i\star)}\| + \sum_{i=1}^{m}\|V_n^{(i\star)}\|=o(1)$. Therefore, we have $\lVert \param_{n} - \param_{n-1} \rVert = o(1)$ by Lemma \ref{lem:retaction_d_upper_bd}. Hence $\param_{n_{k}-1}\rightarrow \param_{\infty}$ as $k\rightarrow\infty$. Therefore since $f$ is continuously differentiable, for each $i=1,\dots, m$,
		\begin{align}
			\lim_{k\to\infty}\grad f_{n_k}^{(i)}(\theta_{n_k-1}^{(i)}) &=  \lim_{k\to\infty} \grad_{i} f(\theta_{n_k}^{(1)},\cdots, \theta_{n_k}^{(i-1)},\theta_{n_k-1}^{(i)},\theta_{n_k-1}^{(i+1)},\cdots,\theta^{(m)}_{n_k-1}) \\&=\grad_{i}f(\theta_{\infty}^{(1)},\cdots, \theta_{\infty}^{(i-1)},\theta_{\infty}^{(i)},\theta_{\infty}^{(i+1)},\cdots,\theta^{(m)}_{\infty})\\
			&=\grad_i f(\param_{\infty}) . 
		\end{align}
		Fix $i\in \{1,\dots,m\}$. For each $\theta\in \Theta^{(i)}$, let $B_{\theta}$ denote the metric ball of radius half of the injectivity radius $\rinj^{(i)}(\theta)$ centered at $\theta$ in $\M^{(i)}$.  Fix $\theta'\in \Theta^{(i)}\cap B_{\param_{\infty}^{(i)}}$ be arbitrary. Note that $\param_{n_{k}-1}\in \Theta^{(i)}\cap B_{\theta_{\infty}^{(i)}} \cap B_{\theta_{n_{k}-1}^{(i)}} $ for all  sufficiently large $k$ since $\theta_{n}^{(i)}\in \Theta^{(i)}$ for all $n\ge 1$ and $\theta_{n_{k}-1}^{(i)}\rightarrow \theta_{\infty}^{(i)}$. Denote by $\Gamma_{k}$ the parallel transport $\Gamma_{\theta_{n_{k}-1}^{(i)}\rightarrow \theta_{\infty}^{(i)}}$ on $\M^{(i)}$. Then 
		\begin{align}
			\left\langle - \Gamma_{k}\left(  \grad f_{n_k}^{(i)}(\theta_{n_{k}-1}^{(i)})  \right) ,\,   \Gamma_{k} \big( \eta_{\theta_{n_{k}}-1}^{(i)}(\theta') \big) \right\rangle   =   \left\langle -\grad f_{n_k}^{(i)}(\theta_{n_{k}-1}^{(i)}) ,\, \eta_{\theta_{n_{k}}-1}^{(i)}(\theta') \right\rangle \le  0.
		\end{align}
		Also note that by the continuity of Riemannian metric, the left-hand side above converges to $\langle -\grad_{i} f(\param_{\infty}) ,\, \eta_{\theta_{\infty}^{(i)}}(\theta') \rangle$, so we obtain 
		\begin{align}
			\langle -\grad_{i} f(\param_{\infty}) ,\, \eta_{\theta_{\infty}^{(i)}}(\theta') \rangle\le  0.
		\end{align}
		Since $\theta'\in \Theta^{(i)}\cap B_{\theta_{\infty}^{(i)}}$ and $i\in \{1,\dots,m\}$ are arbitrary, the above shows that $\param_{\infty}$ is a stationary point of $f$ over $\Param$.

		Lastly, we show \textbf{(iii)}. For each $n\ge 1$ and $i=1,\dots,m$, denote 
			\begin{align}
				\xi_{n}^{(i)}:=\grad f_{n}^{(i)}(\theta_{n-1}^{(i)}) - \grad_{i} f(\param_{n-1}). 
			\end{align}
			Now by writing 
			\begin{align}
				\nabla \hat{g}_{n}^{(i)}(V_{n}^{(i\star)}) -  \grad_{i} \varphi(\param_{n-1}) &= \left( \nabla \hat{g}_{n}^{(i)}(V_{n}^{(i\star)}) - \nabla \hat{g}_{n}^{(i)}(\mathbf{0}) \right) + \left( \grad \varphi_{n}^{(i)}(\theta_{n-1}^{(i)}) - \grad_{i} \varphi(\param_{n-1})  \right)  \\
				&= \zeta_{n}^{(i)} + \xi_{n}^{(i)},
			\end{align}
from \eqref{eq:CRMM_fopt_bd1}, for all $\eta \in T^*_{\theta^{(i)}_{n-1}}$ with $\|\eta\|\le 1$,
		\begin{align}
		\label{eq:CRMM_multi_vi}
			&\langle \grad_i \varphi(\param_{n-1})+\Proj_{T_{\theta^{(i)}_{n-1}}} \partial \psi_n^{(i)}(\theta^{(i)}_{n-1}+V_n^{(i)}),\eta\rangle \\ 
   &= -\langle \xi_{n}^{(i)},\eta\rangle -\langle \zeta_{n}^{(i)},\eta\rangle + \langle \nabla \hat{g}_n (V_n^{(i\star)})+\Proj_{T_{\theta^{(i)}_{n-1}}} \partial \psi_n^{(i)}(\theta^{(i)}_{n-1}+V_n^{(i)}),\eta\rangle \\
			&\ge -(\|\zeta^{(i)}_n\|+\|\xi^{(i)}_n\|)\cdot \|\eta\| -C'\|V_n^{(i\star)}\|-L\|V_n^{(i\star)}\|^2 \\
			&\ge -(\|\zeta^{(i)}_n\|+\|\xi^{(i)}_n\|) -C'\|V_n^{(i\star)}\|-L\|V_n^{(i\star)}\|^2.
			\end{align}
		According to the hypothesis, by using a triangle inequality,  
			\begin{align}
				\lVert \xi_{n}^{(i)} \rVert &\le L'(\lVert V_{n}^{(1)}\rVert + \cdots + \lVert V_{n}^{(i-1)}\rVert) \le L'\sum_{i=1}^{m} \lVert V_{n}^{(i)} \rVert. 
			\end{align}
			Combining with \eqref{eq:CRMM_iter_complexity_bd22},
			\begin{align}
				\lVert \zeta_{n}^{(i)} \rVert+ \lVert\xi_{n}^{(i)} \rVert\le L \lVert V_{n}^{(i\star)} \rVert + L'\sum_{i=1}^{m} \lVert V_{n}^{(i)} \rVert, 
			\end{align}
			where  $\zeta_{n}^{(i)}$ is defined in \eqref{eq:def_zeta_n_i}. From \eqref{eq:tangent_optimal_gap}, we have
   \begin{equation}
       \frac{\rho}{2}\sum_{n=1}^{N}\sum_{i=1}^{m}\|V^{(i)}_n-V_n^{(i\star)}\|^{2} < m\sum_{n=1}^{\infty}\Delta_n <\infty.
   \end{equation}
   Therefore, 
   \begin{equation}
    \sum_{n=1}^{N}\sum_{i=1}^{m}\|V^{(i)}_{n}\|^2 \le 2\sum_{n=1}^{N}\sum_{i=1}^{m}\|V^{(i)}_n-V_n^{(i\star)}\|^{2} + 2\sum_{n=1}^{N}\sum_{i=1}^{m}\|V_n^{(i\star)}\|^{2} < \frac{4\alpha m\sum_{n=1}^{\infty}\Delta_n+4M}{\alpha \rho}.
   \end{equation}
   And also,
    \begin{equation}
\sum_{n=1}^{N}\sum_{i=1}^{m}\|V^{(i)}_{n}\|^2+\sum_{i=1}^{m}\|V^{(i\star)}_{n}\|^2 < \frac{4\alpha m\sum_{n=1}^{\infty}\Delta_n+6M}{\alpha \rho}.
   \end{equation}
   Hence
   \begin{equation}
   \label{eq:CRMM_iter_complexity_bd1_gen1}
       \min_{0\le n\le N}\sum_{i=1}^{m}\|V^{(i)}_{n}\|^2+\sum_{i=1}^{m}\|V^{(i\star)}_{n}\|^2 < \frac{4\alpha m\sum_{n=1}^{\infty}\Delta_n+6M}{\alpha \rho N}.
   \end{equation}
   
   From \eqref{eq:CRMM_iter_complexity_bd1_gen1}, there exists a subsequence $(k_{n})_{n\ge 1}$ such that $1\le k_{n}\le n$ and 
			\begin{align}\label{eq:CRMM_iter_complexity_bd2_gen}
				\max\left\{ \sum_{i=1}^{m}   \lVert V_{k_{n}}^{(i\star)} \rVert^{2},\, \sum_{i=1}^{m}     \lVert V_{k_{n}}^{(i)} \rVert^{2}   \right\} \le    \sum_{i=1}^{m}   \lVert V_{k_{n}}^{(i\star)} \rVert^{2} +   \lVert V_{k_{n}}^{(i)} \rVert^{2} \le  \frac{M'}{\alpha \rho n},
			\end{align}
   where $M':=4\alpha m\sum_{n=1}^{\infty}\Delta_n+6M$.
   
			Hence we get 
			\begin{align}
				(\lVert \zeta_{n}^{(i)} \rVert+ \lVert\xi_{n}^{(i)} \rVert) + C'\|V_n^{(i\star)}\| + L\|V_n^{(i\star)}\|^2
				&\le (L +C')\lVert V_{n}^{(i\star)}\rVert + L' \sum_{i=1}^{m} \lVert V_{n}^{(i)}\rVert+ L\sum_{i=1}^{m}\|V_n^{(i\star)}\|^2  \\
				&\le (L +C')\sqrt{\frac{M'}{\rho\alpha n}} + L'\sqrt{\frac{M'}{\rho \alpha n}} + L\frac{M}{\rho \alpha n} \\
				&= (L+C'+ L') \sqrt{\frac{M'}{\rho \alpha n}}+ L\frac{M}{\rho \alpha n}. 
			\end{align}
	This combining \eqref{eq:CRMM_multi_vi} gives
	\begin{align}
			&\min_{1\le n \le N} \sum_{i=1}^{m}   - \inf_{\eta \in T^*_{\theta^{(i)}_{n-1}},\|\eta\|\le 1}\langle \grad_i \varphi(\param_{n-1})+\Proj_{T_{\theta^{(i)}_{n-1}}} \partial \psi_n^{(i)}(\theta^{(i)}_{n-1}+V_n^{(i)}),\frac{\eta}{\min\{r_0,1\}}\rangle  \\ &\hspace{9cm}\le \frac{L+C' + L'}{\min\{r_0,1\}} \sqrt{\frac{M'}{\rho \alpha N}}+ \frac{L}{\min\{r_0,1\}}\frac{M}{\rho \alpha N}.
		\end{align}
  When $N\ge \frac{L^2}{(L+C'+L')^2}\frac{M^2}{M'\rho\alpha}$, the above gives,
  \begin{align}
			\min_{1\le n \le N} \sum_{i=1}^{m}   - \inf_{\eta \in T^*_{\theta^{(i)}_{n-1}},\|\eta\|\le 1}\langle \grad_i \varphi(\param_{n-1})+\Proj_{T_{\theta^{(i)}_{n-1}}} \partial \psi_n^{(i)}(\theta^{(i)}_{n-1}+V_n^{(i)}),\frac{\eta}{\min\{r_0,1\}}\rangle  \le \frac{2(L+C' + L')}{\min\{r_0,1\}} \sqrt{\frac{M'}{\rho \alpha N}},
\end{align}
	the desired complexity result.
\end{proof}	

\section{Details of Numerical Experiments}
\label{sec:appendix_numerics}

For all the numerical experiments in Section \ref{sec:stylized_app}, we test the algorithms with random initial points. Each experiment is repeated for $10$ times with i.i.d. Gaussian/uniform random initial points, depending on the setting of the problem. The \textit{relative reconstruction error} defined as $\operatorname{Error}(X)=\|X - X^*\|/\|X^*\|$ is computed as a function of elapsed time (on a Macbook Pro 2020 with 1.4 GHz Quad-Core Intel Core i5). Averaged relative reconstruction error with standard deviation are shown by the solid lines and shaded regions in all the plots.

\subsection{Nonnegative Tensor Decomposition with Riemannian Constraints}

For the (nonnegative) tensor decomposition problem with Riemannian constraints we studied in Section \ref{sec:stylized_app}, we pose Riemannian constraints to the first block. Specifically, in the experiments, we let $\M^{(1)}=\mathcal{R}_r \subseteq \R^{50\times 10}$ with $r=2$. The other two blocks are Euclidean spaces, i.e. $\M^{(2)} = \R^{40\times 10}$ and $\M^{(3)} = \R^{30\times 10}$. When the loading matrices are nonnegative, the  nonnegativity can be considered as an additional constraint set within each block. Namely, the constraint sets $\Theta^{(2)}$ and $\Theta^{(3)}$ are nonnegative orthants of the corresponding Euclidean spaces. As for the first low-rank block $M^{(1)}$, it is shown in \cite{song2020nonnegative} that the intersection of the fixed-rank manifold and nonnegative orthant remains a smooth manifold. Therefore, one can also simply let $\M^{(1)}$ to be the nonnegative fixed-rank manifold. The $\ell_1$-regularizer in \eqref{eq:NTF} brings nonsmoothness to the problem. In the numerical experiments, we set the regularization parameter $\lambda_{i}=10^{-2}$ for $i=1, 2, 3$.

\subsection{Regularized Nonnegative Matrix Factorization with Riemannian Constraints}

Similar to the tensor decomposition problem described in the previous section, we let $\M^{(1)}=\mathcal{R}_r \subseteq \R^{50\times 10}$ with $r=5$ and $\M^{(2)}=\R^{10\times 40}$. The nonnegativity can be viewed as constraint sets, as stated in the previous section. This low-rank matrix factorization (without $\ell_1$ regularization) is recently studied in \cite{song2022tangent}, where the authors proposed a tangent space based alternating projection method. In our setting, we further pose the $\ell_1$-regularization term that brings nonsmoothness. The regularization parameter is set to be $\lambda=10^{-2}$.

\subsection{Low-rank Matrix Recovery}
\label{sec:appendix_matrix_rec}
We state the NIHT for solving \eqref{eq:compressed_sensing} in Alg. \ref{alg:NIHT}.

\begin{algorithm}
   \caption{Normalized Iterative Hard Thresholding (NIHT) \cite{tanner2013normalized}}
   \label{alg:NIHT}
\begin{algorithmic}
   \STATE {\bfseries Input:} Initial point $X_0$, left singular vector space $U_0$
   \FOR{$i= 1,\cdots, n$}
   \STATE $\nabla f(X_i) = \mathcal{A}^* (\mathcal{A}(X_i)-b)$
   \STATE $\alpha_i = \frac{\|\Proj_{U_i}(\nabla f(X_i))\|^2}{\|\mathcal{A}\Proj_{U_i}(\nabla f(X_i))\|^2}$
   \STATE $X_{i+1} = \Proj_{\mathcal{R}_r} (X_i - \alpha_i \nabla f(X_i))$
   \ENDFOR
   \STATE \textbf{output:} $X_{n}$
\end{algorithmic}
\end{algorithm}

Here $\mathcal{A}^*: \R^{p}\to \R^{m\times n}$ is the Hermitian adjoint operator of $\mathcal{A}$, $\Proj_{U_i} = U_i U_i^*$ is the projection onto the left singular vector subspace of $X_i$.
Note NIHT is a projected gradient descent algorithm with adaptive step size. 

In the numerical validation in Figure \ref{fig:compressed_sensing}, we set the dimensions to be $m=50$, $n=12$, $p=300$ or $450$ respectively. The true solution $X^*$ is a low-rank matrix in $\R^{m\times n}$ with rank $r=3$. The sensing matrices $A_i$ for $i=1,\cdots, p$ are i.i.d. Gaussian random matrices. The observations $b$ is generated by $b=\mathcal{A}(X^*)$. 

\subsection{Inexact RGD}
We show the performance of inexact RGD based on the low-rank matrix recovery problem in Figure \ref{fig:inexact_RGD}. The dimensions are set to be $m=50$, $n=12$, and $p=300$. The inexactness is posted by adding noise to the Riemannian gradient with $\Delta_n = c/(n+1)^2$, where $n$ is the iteration number. The inexact tBMM is implemented with $c=1$. In the left plot of Figure \ref{fig:inexact_RGD}, inexact tBMM shows similar performance as exact tBMM where both of them outperform NIHT. In the right plot of Figure \ref{fig:inexact_RGD}, it is shown that inexact RGD with different noise parameter $c$ converges at similar speed.

\end{document}